\documentclass[11pt,a4paper]{article}
\usepackage{amsfonts,amsmath,latexsym,verbatim,amscd,mathrsfs,color,array}
\usepackage{amssymb,amsthm}

\usepackage[margin=1in]{geometry}
\usepackage{enumitem}
\usepackage{graphicx}
\usepackage{epstopdf}
\usepackage{cite}
\usepackage{booktabs}
\usepackage{float}
\usepackage{multirow, multicol}

\usepackage{caption2}
\usepackage{subfigure}

\usepackage[colorlinks=true,linkcolor=blue,citecolor=blue,urlcolor=blue]{hyperref}

\newtheorem{theorem}{Theorem}[section]
\newtheorem{lemma}{Lemma}[section]

\newtheorem{prop}{Proposition}[section]
\newtheorem{remark}{Remark}[section]

\newcommand{\R}{\mathbb{R}}
\newcommand{\norm}[1]{\left\lVert #1 \right\rVert}

\title{Data-Driven optimal control via Koopman operators and Hamilton-Jacobi-Bellman equations}
\author{Guoyuan Chen$^1$\thanks{$^1$School of Data Sciences, Zhejiang University of Finance and Economics, Hangzhou 310018, P. R. China, E-mail: gychen@zufe.edu.cn} and Gaosheng Zhu$^3$\thanks{$^3$School of Mathematics, Harbin Institute of Technology, Harbin 150001, P. R. China, E-mail: gaozsc@163.com}}
\date{}

\begin{document}
\maketitle

\begin{abstract}
This paper presents a data-driven stable manifold (DD-SM) method, which integrates Koopman operator representation learning with the geometric stable manifold approach to Hamilton-Jacobi-Bellman (HJB) equations, enabling end-to-end optimal feedback control synthesis from raw trajectory data without prior knowledge of system dynamics.

We construct an augmented control system under a unified symmetric subspace decomposition (SSD) and extended dynamic mode decomposition (EDMD) framework for joint approximation of the drift field, control matrix and their spatial derivatives, and derive probabilistic finite-sample error bounds for invariant and non-invariant dictionary spaces to yield a provably accurate approximate characteristic system of HJB equation.
Via Lyapunov-Perron operator and ODE perturbation analysis, we prove the data-driven stable manifold achieves monotonically decreasing semi-global error with growing training data. We further establish closed-loop exponential stability and quantify the optimality gap, both tightenable by refining model accuracy.
An efficient algorithm pipeline with adaptive data generation and deep neural approximation is developed, outputting control signals within 1 millisecond. Experiments on a modified van der Pol oscillator verify the effectiveness of our method.

\end{abstract}

\section{Introduction}

Data-driven modeling and control of nonlinear systems have become a core research direction bridging machine learning, control theory and applied mathematics. Learning-based methods effectively address challenging nonlinear, high-dimensional and unmodeled dynamics scenarios where traditional model-based techniques fail. A fundamental challenge in this field is to develop optimal control strategies that combine rigorous theoretical performance guarantees with practical engineering applicability, particularly for infinite-horizon optimal control problems formulated via HJB equation.

For infinite horizon optimal control, the stabilizing HJB solution corresponds to the stable manifold of its characteristic Hamiltonian system \cite{van1991state, sakamoto2008analytical}. Optimal feedback can be derived directly from this manifold, which generalises LQR theory for nonlinear systems. Traditional stable manifold solvers are confined to equilibrium neighbourhoods, plagued by the curse of dimensionality and unquantifiable approximation error \cite{sakamoto2008analytical,sakamoto2013case,chen2020symplectic}. Though recent deep learning schemes support mesh-free high dimensional manifold fitting \cite{chen2024deep,chen2025h}, they require known system dynamics and lack rigorous guarantees for pure data-driven control from raw trajectory measurements.

First proposed a century ago, the Koopman operator framework transforms nonlinear system investigation into linear operator analysis on function spaces \cite{koopman1931hamiltonian, koopman1932dynamical}. It captures state space dynamics through observable function evolution and bypasses intractable nonlinear equations \cite{mezic2005spectral, rowley2009spectral, mezic2013analysis}. As the prevailing data-driven approximation method, extended dynamic mode decomposition (EDMD) lifts states into a feature space spanned by basis observables and computes a finite dimensional Koopman operator projection via least squares regression \cite{Brunton2022Koopman}.

\subsection{Related Work}

Observable dictionary construction is central to EDMD-based Koopman approximation. A dictionary spanned by Koopman eigenfunctions forms an invariant subspace of the operator \cite{brunton2016koopman}, reducing nonlinear system to linear operator on this subspace. This has driven extensive research on Koopman invariant subspaces \cite{takeishi2017learning}, yielding metrics such as the consistency index to balance model expressiveness and accuracy \cite{haseli2023d, haseli2023invariance}.For automatic dictionary construction, various methods have been developed, including analytical constructions \cite{shi2021acd}, Koopman eigenfunction kernels \cite{bevanda2023kernel}, and deep neural network approaches \cite{li2017extended, lusch2018deep, yeung2019deep}. Extensions to controlled systems \cite{korda2020optimal}, joint state reconstruction schemes for stable systems \cite{bevanda2022dictionary}, and physics-informed EDMD integrating prior knowledge \cite{baddoo2023physics} have also been proposed. For dictionary optimization, symmetric subspace decomposition (SSD) type methods identify maximal invariant subspaces to improve modeling accuracy \cite{haseli2021learning, haseli2023generalizing, haseli2023invariance}. \cite{conradie2026trustworthy} presents a unified a posteriori framework to validate and refine Koopman approximations.
Rigorous spectral analysis of the Koopman operator remains challenging due to its generally continuous spectrum. Provable DMD variants include residual DMD \cite{colbrook2024resdmd} and rigged DMD for resolvent/continuous spectrum estimation \cite{colbrook2025rigged}. Sharp convergence rates \cite{kostic2023sharp}, fundamental approximability limits \cite{colbrook2024barriers}, and Liouville-based singular DMD using generator compactness \cite{rosenfeld2023singular} have been established. Further relevant works can be found in \cite{conradie2026trustworthy, strasser2026overview}.

Finite-data EDMD error bounds have advanced notably, with total error decomposed into projection error (finite dictionary truncation) and statistical estimation error (limited samples).Building on infinite-limit strong convergence \cite{korda2018convergence}, Mezi\'{c} derived a probabilistic finite-sample bound for deterministic ergodic systems with discrete non-dense spectra \cite{mezic2022numerical}. Zhang and Zuazua unified bounds for both error components for i.i.d.-sampled systems via finite-element methods and associated linear transport equation \cite{zhang2023quantitative}. For both ordinary and stochastic systems, \cite{nuske2023finitedata} established probabilistic bounds with quadratic sample complexity, extending to generators and controlled systems under exponential stability requirements; \cite{philipp2024variance} relaxed these assumptions via variance representations for tighter bounds.
\cite{llamazares2024datadriven} proved the convergence of the approximating operator with explicit convergence rate under relaxed conditions. In \cite{kohne2025lerror, bold2025kernel}, kernel-based EDMD delivers uniform pointwise bounds with equilibrium-vanishing proportional errors, enabling rigorous closed-loop control guarantees. Further relevant results can be found in \cite{conradie2026trustworthy, strasser2026overview}.

Data-driven control has advanced rapidly, with methods that learn surrogate models or synthesize controllers directly from observed trajectories (see, e.g., \cite{berberich2024overview, depersis2023learning, martin2023sdpsurvey}).
The Koopman operator framework is a prominent paradigm for data-driven nonlinear control. Extended dynamic mode decomposition with control (EDMDc) constructs linear state-input representations in the lifted space, enabling classical control methods such as LQR and MPC (e.g. \cite{proctor2016dynamic, korda2018linear, iacob2024koopman}).
Several bilinear EDMDc variants have also been established (see e.g., \cite{bruder2021advantages, strasser2024safedmd, strasser2025kernel}). For a thorough review on details of EDMDc, we refer the reader to \cite{strasser2026overview}.

Numerous existing studies on Koopman-based optimal control integrate Koopman operator techniques to develop data-driven model predictive control schemes \cite{Arbabi2018Data, korda2018linear, Peitz2020Data, Zhang2022Robust, Mamakoukas2022Robust, Yuan2022data, Narasingam2023Data, deJong2024Koopman, bold2025data, bold2025kernel, schimperna2025data, Xie2025Bilinear, Cibulka2025dictionary, Chen2026data}.
Nevertheless, one crucial research gap remains for general optimal control problems formulated via the HJB equation. Existing work lacks systematic analysis on how Koopman approximation error propagates toward approximate HJB solutions, and rigorous theoretical assurances for closed-loop stability and control optimality are therefore still absent.


\subsection{Contributions}

To address the above challenges, this paper proposes the \textit{data-driven stable manifold (DD-SM)} method, which integrates Koopman operator representation learning with the stable manifold approach for HJB equations. The main contributions are summarized as follows:

(1) We introduce an augmented control system by differentiating the original dynamics, enabling joint approximation of the drift field, control matrix and their spatial derivatives. A unified EDMD-SSD framework is proposed to identify approximate Koopman-invariant subspaces of the augmented system, with rigorous probabilistic error bounds derived for both invariant and non-invariant dictionary spaces (Theorems \ref{t:bounds-aug} and \ref{t:bounds-aug-general}), yielding a provably accurate approximate characteristic Hamiltonian system for the HJB problem.

(2) Through perturbation analysis of Lyapunov-Perron operators and ODEs, we prove that the error between the data-driven approximate stable manifold and its exact counterpart can be made arbitrarily small over a semi-global domain, decreases monotonically with growing training data volume, and establishes a quantitative link between data quantity and control performance.

(3) Building on the manifold error bounds, we prove that the feedback controller constructed from the approximate stable manifold ensures exponential stability of the closed-loop system at the equilibrium point. We further quantify the gap between the data-driven control cost and the exact optimal cost, showing that both the stability margin and the optimality gap can be systematically tightened by refining model approximation accuracy.

(4) Guided by our theoretical results, we develop a computationally efficient algorithm pipeline integrating adaptive data generation, model refinement, and deep neural network approximation for stable manifold computation. The trained network generates control signals within 1 millisecond, satisfying real-time requirements of practical engineering systems.

At the core of our analysis is a complete error propagation chain connecting finite training data, Koopman model approximation, stable manifold computation, and closed-loop control performance. The proof proceeds in three stages: first, tight error bounds for the augmented system are derived via the unified SSD-EDMD framework, establishing linear-in-$y$ observables over a semi-global domain; second, the classical Lyapunov-Perron method is extended to the perturbed approximate Hamiltonian system for local manifold error bounds, which are then extended to the semi-global domain via ODE perturbation theory; third, manifold error bounds are translated into explicit guarantees for closed-loop exponential stability and near-optimal control cost.

\subsection{Organization}
The rest of the paper is organized as follows. Section \ref{s:preliminaries} reviews Koopman and Lie operators. Section \ref{s:DD-control} proves some basic estimates for data-driven Koopman approximation. Section \ref{s:EDMD} constructs the augmented control system and presents EDMD-based approximation with error bounds. Section \ref{s:stable} reviews stable manifold method. Section \ref{s:edmd-sm} develops data-driven stable manifolds with semi-global error bounds, and analyzes closed-loop stability and optimality. Section \ref{s:algorithm} details the algorithm pipeline. Section \ref{s:example} validates the DD-SM method on a modified van der Pol optimal control problem.

\subsection{Notation}
\label{sec:intro_notations}

Let $\mathbb{R}$ denote the set of real numbers. Let $[l:n]$ denote the set of integers $i$ satisfying $l \le i \le n$. Let $E_i = (0, \cdots, 0, 1, 0, \cdots, 0)$ denote the standard unit vector with the $i$-th component equal to 1, for $i \in [1:n]$. Let $\Omega \subset \mathbb{R}^n$ be a bounded domain with piecewise smooth boundary, and let $\bar{\Omega}$ denote its closure. Let $W^{1,2}(\Omega) := \{\varphi \in L^2(\Omega) \mid \int_{\Omega} |\nabla\varphi(x)|^2 dx < \infty\}$ be the $(1,2)$ Sobolev space with norm
$
\|\varphi\| = \left[\int_{\Omega} \left(|\varphi(x)|^2 + |\nabla \varphi(x)|^2\right) dx\right]^{\frac{1}{2}}.
$
For any matrix $A$, $A^{\dagger}$ denotes its Moore--Penrose pseudo-inverse.

\section{Preliminaries}\label{s:preliminaries}

This section briefly reviews Koopman operators for control systems; further details can be found in \cite{Brunton2022Koopman}.

Let $\mathbf{\Omega}\subset\mathbb R^J$ be a bounded domain with piecewise smooth boundary.
Consider the nonlinear system
\begin{eqnarray}\label{e:system-g}
\dot \xi(t)=F(\xi(t)),
\end{eqnarray}
where $F:\mathbb R^J\to \mathbb R^J$ is locally Lipschitz and satisfies $F(0)=0$. Let $\xi(t;\xi_0)$ stand for the unique system trajectory starting from $\xi(0)=\xi_0\in\\mathbf{Omega}$ over $t\ge0$. The domain $\mathbf{\Omega}$ is assumed forward invariant under the system flow, i.e., $\xi(t;\xi_0)\in \mathbf{\Omega}$ holds for all $\xi_0\in\mathbf{\Omega}$ and $t\in[0,\infty)$.

The Koopman operator of system \eqref{e:system-g} is defined by
$
(\mathcal K^t\varphi)(\xi_0)=\varphi(\xi(t;\xi_0)), t\ge0,\xi_0\in \mathbf{\Omega},\varphi\in L^2(\mathbf{\Omega}),
$
where $\varphi$ denotes observables.
Its infinitesimal generator for $\varphi\in W^{1,2}(\mathbf{\Omega})$ reads
$
\mathcal L\varphi =\lim_{t\to 0^+}\frac{\mathcal K^t\varphi -\varphi}{t},
$
which simplifies via direct calculation to the Lie operator
\begin{eqnarray}\label{e:lie-op-comp}
(\mathcal L\varphi)(\xi)=(F\cdot \nabla \varphi)(\xi).
\end{eqnarray}
Setting $\Phi(t)=\mathcal K^t\varphi=\varphi(\xi(t;\cdot))$, we obtain the evolution equation $\dot\Phi(t)=\mathcal L\Phi(t)$ with initial condition $\Phi(0)=\varphi$.

Let $\mathbb U\subset\mathbb R^M$ denote a domain with piecewise smooth boundary. We study the control-affine system
\begin{eqnarray}\label{e:control-g}
\dot \xi=F(\xi)+G(\xi)U,
\end{eqnarray}
where $U(t)=(U_1,\dots,U_M)^\top\in L_{\rm loc}^{\infty}([0,\infty),\mathbb U)$ stands for control inputs. The drift term $F\in C^2(\mathbf{\Omega})$ satisfies $F(0)=0$, and each column $G_i$ of $G$ also lies in $C^2(\mathbf{\Omega})$.

$\mathcal K^t_U$ denotes the Koopman operator for the system flow under constant control $U$, with its associated Lie operator written as $\mathcal L^U$. Define $E_i\in\mathbb R^M$ as the standard unit vector with unity at the $i$-th entry; $\mathcal L^0$ and $\mathcal L^{E_i}$ correspond to zero input and unit input $E_i$, respectively. Combining Eq. \eqref{e:lie-op-comp} and system dynamics \eqref{e:control-g} (see e.g. \cite{Surana2016koopman}), we derive
\begin{equation}\label{e:lie-op-Lu}
\mathcal L^U=\mathcal L^0+\sum_{i=1}^{M}U_i(\mathcal L^{E_i}-\mathcal L^0),\quad \forall U\in\mathbb R^{M}.
\end{equation}

\section{Data-driven control system based on Koopman operator }\label{s:DD-control}
This section derives Koopman-based error bounds for data-driven control systems. We use notations as in \cite{bold2025data, strasser2024koopman}.

Let $\mathbb V=\mathrm{span}\{\psi_k\mid 0\le k\le N\}$ be an $N$-dimensional dictionary subspace spanned by observables $\psi_k\in W^{1,2}(\mathbf{\Omega})$. $P_{\mathbb V}: L^2(\mathbf{\Omega})\to\mathbb V$ denotes the $L^2$ orthogonal projection onto $\mathbb V$. Assume $d$ independent and identically distributed (i.i.d.) data points $\omega_1,\cdots, \omega_d\in \mathbf{\Omega}$, define the $N\times d$-matrices
\begin{equation}\label{e:X_0}
X^0:=\left[\begin{array}{ccc}
           \psi_0(\omega_1) & \cdots & \psi_0(\omega_d) \\
           \vdots & \vdots & \vdots \\
           \psi_N(\omega_1) & \cdots & \psi_N(\omega_d)
         \end{array}
\right],
\end{equation}
\begin{equation}\label{e:Y_0}
Y^0:=\left[\begin{array}{ccc}
           \mathcal L^0\psi_0(\omega_1) & \cdots & \mathcal L^0\psi_0(\omega_d) \\
           \vdots & \vdots & \vdots \\
           \mathcal L^0\psi_N(\omega_1) & \cdots & \mathcal L^0\psi_N(\omega_d)
         \end{array}
\right],
\end{equation}
where $(\mathcal L^0\psi_k)(\omega_j)=(F\cdot\nabla \psi_k)(\omega_j)$ for $k\in [0:N]$ and $j\in [1:d]$. Define EDMD estimator of $\mathcal L^0$ as
\begin{equation*}
\mathcal L^0_d := {\rm arg} \min_{L\in \mathbb R^{N\times N}}\|LX^0-Y^0\|_F^2,
\end{equation*}
where $\mathbb R^{N\times N}$ denotes the set of real matrices of form $N\times N$, $\|\cdot\|_F$ is the Frobenius norm of a matrix.

Let
$
(\mathcal L^{E_i}\psi_k)(\omega_j)=\nabla \psi_k(\omega_j)^T(F(\omega_j)+G(\omega_j)E_i).
$
Define EDMD estimator of $\mathcal L^{E_i}$ as
\begin{equation*}
\mathcal L^{E_i}_d := {\rm arg} \min_{L\in \mathbb R^{N\times N}}\|LX^{E_i}-Y^{E_i}\|_F^2,
\end{equation*}
where
\begin{equation}\label{e:X_i}
X^{E_i}:=\left[\begin{array}{ccc}
           \psi_0(\omega_1) & \cdots & \psi_0(\omega_d) \\
           \vdots & \vdots & \vdots \\
           \psi_N(\omega_1) & \cdots & \psi_N(\omega_d)
         \end{array}
\right],
\end{equation}
\begin{equation}\label{e:Y_i}
Y^{E_i}:=\left[\begin{array}{ccc}
           \mathcal L^{E_i}\psi_0(\omega_1) & \cdots & \mathcal L^{E_i}\psi_0(\omega_d) \\
           \vdots & \vdots & \vdots \\
           \mathcal L^{E_i}\psi_N(\omega_1) & \cdots & \mathcal L^{E_i}\psi_N(\omega_d)
         \end{array}
\right].
\end{equation}

In the following, we choose a fixed dictionary of form
$
\Phi(\xi)=[1,\xi^T,\phi_{J+1},\cdots,\phi_N(\xi)]^T,
$
where $\phi_0\equiv 1$, $\phi_k(\xi)=\xi_k$, $k\in [1:J]$, and $\phi_k\in C^2(\mathbf{\bar\Omega},\mathbb R)$ satisfy $\phi_k(0)=0$, $k\in [n+1: N]$. Define the dictionary as $\mathbb V:={\rm span}\{\phi_k, k\in[0:N]\}$. We assume
$
\|\Phi(\xi)-\Phi(0)\|\le L_\Phi\|\xi\|, \quad \forall \xi\in \mathbf{\Omega},
$
where $L_{\Phi}$ is constant depending only on $\Phi$.

\subsection{Special case: $\mathbb V$ is invariant under $\mathcal K^t_U$}

First, let us consider the case that $\mathbb V$ is invariant with respect to the Koopman operator $\mathcal K^t_U$. The general case will be discussed in Section \ref{s:general}.

\textbf{Assumption 1} (Invariance of the dictionary $\mathbb V$). For any $\phi\in \mathbb V$, it holds that $\phi(\xi(t;\cdot,U))\in \mathbb V$ for all $U(t)\equiv U\in \mathbb U $  and $t\ge 0$, that is, $\mathcal K_U^t \mathbb V\subset \mathbb V$.

\begin{remark}
Assumption 1 is equivalent to that for all $U\in \mathbb U$,
$
P_{\mathbb V}\mathcal L^U|_{\mathbb V}=\mathcal L^U|_{\mathbb V},
$
\end{remark}

Since the space $\mathbb V$ is finite dimensional, under Assumption 1, the Lie operator restricted on $\mathbb V$ can be represented as a matrix. Specifically,
noting that $\phi_0(\xi)\equiv 1$,
$
\frac{d}{dt}\phi_0(\xi(t;\cdot,U))\equiv 0.
$
Using $F(0)=0$, we have that
$
(\mathcal L^0 \phi_i)(0)=\nabla \phi_i(0)\cdot F(0)=0.
$
Hence $(\mathcal L^0 \phi_i)(\xi)=0 \phi_0+a_{i1}\phi_1(\xi)+\cdots + a_{iN}\phi_N(\xi)$.  That is, on the basis $\Phi(\xi)=[1,\phi_1,\cdots,\phi_N]$,
\begin{eqnarray}\label{e:L0}
\mathcal L^0= \left[
                \begin{array}{cc}
                  0 & 0_{1\times N} \\
                  0_{N\times 1} & \mathcal L^0_{22} \\
                \end{array}
              \right],
\end{eqnarray}
where $\mathcal L^0_{22}\in \mathbb R^{N\times N}$. Similarly, represent $\mathcal L^{E_i}$, $i\in [1:M]$, as
\begin{eqnarray}\label{e:Lei}
\mathcal L^{E_i}=\left[
                \begin{array}{cc}
                  0 & 0_{1\times N} \\
                  \mathcal L^{E_i}_{21} & \mathcal L^{E_i}_{22} \\
                \end{array}
              \right],
\end{eqnarray}
where $\mathcal L^{E_i}_{21}\in \mathbb R^{N\times 1}$ and $\mathcal L^{E_i}_{22}\in \mathbb R^{N\times N}$.

We adopt EDMD to construct a data-driven approximation $\mathcal L_d^U$ of the Lie operator $\mathcal L^U$ \cite{strasser2024koopman}.
For each constant input $\bar U\in\{0,E_1,\dots,E_M\}$ with $d^{\bar U}$ samples in $\mathbf{\Omega}$, we build the dictionary matrix $X^{\bar U}$ and the corresponding Lie-derivative matrix $Y^{\bar U}$ from the basis $\Phi(\xi)$. The EDMD estimate of $P_{\mathbb V}\mathcal L^{\bar U}|_{\mathbb V}$ is obtained via Frobenius-norm least-squares fitting, yielding the block-structured forms
\begin{eqnarray}\label{e:EDMD-Lu}
\mathcal L^0_d=\begin{bmatrix}0 & 0_{1\times N} \\ 0_{N\times 1} & S\end{bmatrix},\quad \mathcal L^{E_i}_d=\begin{bmatrix}0 & 0_{1\times N} \\ B_{0,i} & \hat B_i\end{bmatrix},
\end{eqnarray}
where $i\in [1:M]$, $S=Y^0 (X^0)^{\dag}$ and $[B_{0,i},\hat B_i]=Y^{E_i}(X^{E_i})^{\dag}$ are closed-form pseudo-inverse solutions. Using \eqref{e:lie-op-Lu}, the control-dependent operator approximation is given by
\begin{equation}\label{e:lie-op-app-2}
\mathcal L^U_d=\mathcal L^0_d+\sum_{i=1}^{M}U_i(\mathcal L^{E_i}_d-\mathcal L^0_d).
\end{equation}

The finite-sample error bound is:
\begin{theorem}[\cite{schaller2023towards}]\label{t:app-l-ld}
Under Assumption 1 with i.i.d. samples, for any error tolerance $c_r>0$ and confidence level $1-\delta$, there exists a sample size threshold $d_0=O(1/(c_r^2\delta))$ such that for all $d\ge d_0$ and $U\in\mathbb U$,
\begin{eqnarray}\label{e:error-data}
\|\mathcal L^U|_{\mathbb V}-\mathcal L^U_d\|\le \eta
\end{eqnarray}
holds with probability $1-\delta$.
\end{theorem}

To build a data-driven approximate control system, we analyze the projection error from the dictionary space $\mathbb V$ back to the state space $\mathbb R^J=\mathrm{span}\{\xi_1,\dots,\xi_J\}$. A preliminary error bound was given in \cite[Proposition 5]{strasser2024koopman}; we refine it for tighter state-projection estimates suited to optimal control analysis, yielding the following result.

\begin{prop}\label{p:app}
Under Assumption 1 with i.i.d. samples, for any $\delta\in(0,1)$ and $\eta>0$, there exists a sample threshold $d_0=O(\frac{1}{\eta^2\delta})$ such that for all $d\ge d_0$, system \eqref{e:system-g} admits the data-driven representation
\begin{eqnarray}
\dot \xi=\left(F^{\eta}(\xi) + \alpha(\xi)\right)+ \left(G^{\eta}(\xi)+\beta(\xi)\right)U,
\end{eqnarray}
where the residual terms satisfy
\begin{eqnarray}\label{e:f-g-error}
\|\alpha(\xi)\|\le \eta \|\xi\|,\quad \|\beta(\xi)\|\le \eta,\quad \forall \xi\in \mathbf{\Omega},
\end{eqnarray}
with probability $1-\delta$.
\end{prop}

\begin{proof}
We first compute
$
\frac{d}{dt}\Phi(\xi(t))=\mathcal L^U\Phi(\xi(t)),
$
where $U\in \mathbb U$.
Let $\hat \Phi(\xi)=[0_{N\times 1}, I_N]\Phi(\xi)$. Then by \eqref{e:L0} and \eqref{e:Lei}, we have that, for $i\in[1:M]$,
\begin{eqnarray}
\mathcal L^0 \Phi(\xi)=\left[
                       \begin{array}{c}
                         0\\
                         \mathcal L^0_{22}\hat\Phi(\xi) \\
                       \end{array}
                     \right], \quad \mathcal L^{E_i} \Phi(\xi)=\left[
                       \begin{array}{c}
                         0\\
                         \mathcal L^{E_i}_{21}+\mathcal L^{E_i}_{22}\hat\Phi(\xi) \\
                       \end{array}
                     \right].\notag
\end{eqnarray}
Let $\mathcal L^E_{21}=[\mathcal L^{E_1}_{21},\cdots,\mathcal L^{E_M}_{21}]$ and $\tilde{\mathcal L}^{E_i}_{22}=\mathcal L^{E_i}_{22}-\mathcal L^0_{22}$. Then, from \eqref{e:lie-op-Lu}, we obtain that
\begin{eqnarray}
\mathcal L^{U} \Phi(\xi)=\begin{bmatrix}
                         0\\
                         \mathcal L^0_{22}\hat\Phi(\xi)+ \mathcal L^{E}_{21}U+\sum_{i=1}^M U_i\tilde{\mathcal L}^{E_i}_{22}\hat\Phi(\xi) \\
                       \end{bmatrix}.\notag
\end{eqnarray}
From \eqref{e:EDMD-Lu}, the EDMD approximation has form
\begin{eqnarray}
\mathcal L^{U}_d \Phi(\xi)
                     =\begin{bmatrix}
                         0\\
                         S\hat\Phi(\xi) \\
                       \end{bmatrix}
                     +\begin{bmatrix}
                         0\\
                         B_0U+\sum_{i=1}^M U_i B_i\hat\Phi(\xi) \\
                       \end{bmatrix},\notag
\end{eqnarray}
where $B_0=[B_{0,1},\cdots, B_{0,M}]$ and $B_i=\hat B_i-S$. From the definition of $\Phi(\xi)$,  define
$
P_{\xi}:\mathbb R^N\to \mathbb R^J={\rm span}\{\xi_1,\cdots, \xi_J\}
$
to be the projection of vectors in $\mathbb V$ to the $[1:J]$ components. That is, $P_{\xi}\Phi(\xi)=\xi$. Then we define
\begin{eqnarray}\label{e:f-g-app}
F^{\eta}(\xi)=P_\xi
                       \begin{bmatrix}
                         0\\
                         S\hat\Phi(\xi) \\
                       \end{bmatrix},\quad
G^{\eta}(\xi)=P_\xi
                       \begin{bmatrix}
                         0\\
                         B_0+ B(\xi) \\
                       \end{bmatrix},
\end{eqnarray}
where $B(\xi)=[B_1\hat\Phi(\xi),\cdots, B_M\hat\Phi(\xi)]$.
Furthermore,
\begin{eqnarray}
&&(\mathcal L^U-\mathcal L^U_d)\Phi(\xi)=
     \begin{bmatrix}
       0 \\
       (\mathcal L^0_{22}-S)\hat\Phi(\xi) \\
     \end{bmatrix}\notag\\
&&~~~~+\begin{bmatrix}
                         0\\
                         (\mathcal L^{E}_{21}-B_0)U+\sum_{i=1}^M U_i(\tilde{\mathcal L}^{E_i}_{22}-B_i)\hat\Phi(\xi) \\
                       \end{bmatrix}.\notag
\end{eqnarray}
Then we have the errors
\begin{align*}
\alpha(\xi)&=F(\xi)-F^{\eta}(\xi)=P_\xi
     \begin{bmatrix}
       0 \\
       (\mathcal L^0_{22}-S)\hat\Phi(\xi), \\
     \end{bmatrix}\notag\\
\beta(\xi)U&= (G(\xi)-G^{\eta}(\xi))U\notag\\
&=P_\xi
                       \begin{bmatrix}
                         0\\
                         (\mathcal L^{E}_{21}-B_0)U+\sum_{i=1}^M U_i(\tilde{\mathcal L}^{E_i}_{22}-B_i)\hat\Phi(\xi) \\
                       \end{bmatrix}.\notag
\end{align*}
Using Theorem \ref{t:app-l-ld}, we get that, for any given $c_r>0$, there is a constant $d_0=O(\frac{1}{c_r^2\delta})$ such that
\begin{eqnarray}\label{e:l-error1}
&&\|(\mathcal L^0_{22}-S)\hat\Phi(\xi)\|\le c_rL_{\Phi}\|\xi\|,\\
&&\|(\mathcal L^{E}_{21}-B_0)U\|\le c_r\|U\|,\notag\\
&&\left\|\sum_{i=1}^M U_i(\tilde{\mathcal L}^{E_i}_{22}-B_i)\hat\Phi(\xi)\right\|\le c_rL_{\Phi}\|\xi\|\|U\|,\notag
\end{eqnarray}
for all $(\xi,U)\in \mathbf{\Omega}\times \mathbb U$ with probability $1-\delta$. Let $C_{\mathbf{\Omega}}$ be the diameter of the domain $\mathbf{\Omega}$.
Therefore, letting $c_r= \min\left\{\frac{\eta}{L_{\Phi}}, \frac{\eta}{2}, \frac{\eta}{2L_{\Phi}C_{\mathbf{\Omega}}}\right\}$, there is $d_0=O(\frac{1}{\eta^2\delta})$ such that for all $d\ge d_0$, \eqref{e:l-error1} and \eqref{e:f-g-error} hold.
\end{proof}

\subsection{General case: $\mathcal{L}^U \mathbb{V} \nsubseteq \mathbb{V}$}\label{s:general}
Assumption 1 requires $\mathbb{V}$ to be invariant under the control-dependent Koopman operator $\mathcal{K}_U^t$ for all $U\in\mathbb{U}$ and $t>0$. Under invariance, the EDMD approximation error of $\mathcal{L}^U$ stems solely from finite data; when invariance is violated, an additional projection error arises as $\mathcal{K}_U^t \phi$ cannot be exactly represented within $\mathbb{V}$.

Analogous to the discrete Koopman setting in \cite{haseli2023invariance}, we define an invariance index for the Lie operator $\mathcal{L}^U$ on a finite-dimensional subspace $\mathbb{V}\subset L^2(\mathbf{\Omega})$:
$
\mathcal{I}_{\mathcal L^U}(\mathbb{V}) = \sup_{\phi \in  \mathbb{V}, \|\mathcal L^U\phi\| \neq 0} \frac{\|\mathcal L^U\phi - P_{ \mathbb{V}} \mathcal L^U\phi\|}{\|\mathcal L^U\phi\|}.
$
By Theorem \ref{t:index-sin} below, which generalizes \cite[Theorem 5.1]{haseli2023invariance},
$
\mathcal{I}_{\mathcal L^U}(\mathbb{V}) = \sin\theta_{\text{max}},
$
with $\theta_{\text{max}}$ the largest Jordan principal angle between $\mathbb{V}$ and $\mathcal{L}^U\mathbb{V}$. The index quantifies the degree of invariance violation: $\mathcal{I}_{\mathcal L^U}(\mathbb{V})=0$ if and only if $\mathbb{V}$ is $\mathcal{L}^U$-invariant.

We now refine the error bounds in Theorem~\ref{t:app-l-ld} and Proposition~\ref{p:app} by incorporating the invariance violation error characterized by $\mathcal{I}_{\mathcal{L}^U}(\mathbb{V})$.
For Theorem~\ref{t:app-l-ld}, the total error under non-invariant dictionaries decomposes into projection error and statistical error:
\begin{eqnarray}\label{e:non-inv}
\|\mathcal{L}^U|_{\mathbb V} - \mathcal{L}_d^U\| \leq \underbrace{\|\mathcal{L}^U|_{\mathbb V} - \mathcal{P}_{\mathbb{V}} \mathcal{L}^U|_{\mathbb V}\|}_{\text{Projection Error}} + \underbrace{\|\mathcal{P}_{\mathbb{V}} \mathcal{L}^U|_{\mathbb V} - \mathcal{L}_d^U\|}_{\text{Statistical Error}}.
\end{eqnarray}
The projection error is bounded by the invariance index:
\begin{eqnarray}\label{e:proj-error}
\|\mathcal{L}^U|_{\mathbb V} - \mathcal{P}_{\mathbb{V}} \mathcal{L}^U|_{\mathbb V}\| \leq \|\mathcal{L}^U|_{\mathbb{V}}\| \cdot \mathcal{I}_{\mathcal{L}^U}(\mathbb{V}),
\end{eqnarray}
where $\|\mathcal{L}^U|_{\mathbb{V}}\|$ denotes the bounded restricted operator norm. Combined with the statistical error $c_r$, the refined bound reads
$
\|\mathcal{L}^U - \mathcal{L}_d^U\| \leq \|\mathcal{L}^U|_{\mathbb{V}}\| \cdot \mathcal{I}_{\mathcal{L}^U}(\mathbb{V}) + c_r.
$

For Proposition~\ref{p:app}, the state-space approximation errors $\alpha(\xi)$ and $\beta(\xi)$ similarly inherit projection error contributions. Decomposing the drift error as
\begin{eqnarray}
\alpha(\xi) = \underbrace{P_{\xi} \left[ (\mathcal{L}^U - \mathcal{P}_{\mathbb{V}} \mathcal{L}^U)\Phi(\xi) \right]}_{T_1} + \underbrace{P_{\xi} \left[ (\mathcal{P}_{\mathbb{V}} \mathcal{L}^U - \mathcal{L}_d^U)\Phi(\xi) \right]}_{T_2}.\notag
\end{eqnarray}
Since $\|T_1\|\leq \|\mathcal{L}^U|_{\mathbb{V}}\| \mathcal{I}_{\mathcal{L}^U}(\mathbb{V}) L_{\Phi} \|\xi\|$, it holds that
$
\|\alpha(\xi)\| \leq \left( \|\mathcal{L}^U|_{\mathbb{V}}\| \mathcal{I}_{\mathcal{L}^U}(\mathbb{V}) L_{\Phi} + \eta \right) \|\xi\|.
$
The control matrix error follows analogously:
$
\|\beta(\xi)\| \leq \|\mathcal{L}^U|_{\mathbb{V}}\| \mathcal{I}_{\mathcal{L}^U}(\mathbb{V}) + \eta.
$

\begin{remark}
Exact evaluation of $\|\mathcal{L}^U|_{\mathbb{V}}\|$ and $\mathcal{I}_{\mathcal{L}^U}(\mathbb{V})$ requires full knowledge of $\mathcal{L}^U$, which is infeasible for unknown systems. See \cite{conradie2026trustworthy} for related results.
Our strategy fixes a sufficiently large finite-dimensional dictionary $\tilde{\mathbb{V}}$, for which $\|\mathcal{L}^U|_{\tilde{\mathbb{V}}}\|$ is a well-defined finite constant. We then select a subspace $\mathbb{V} \subset \tilde{\mathbb{V}}$ to minimize the invariance index, which can be reliably estimated from trajectory data \cite{haseli2021learning}. Hence the projection error can be made sufficiently small on the approximate invariant $\mathbb{V}$ for real applications.
\end{remark}

\section{EDMD for augmented control system}\label{s:EDMD}

This work investigates data-driven optimal control for unknown control-affine systems
\begin{eqnarray}\label{e:control}
\dot{x} = f(x) + g(x) u,
\end{eqnarray}
where $f\in C^2(\Omega)$ with $f(0)=0$ and $g=[g_1,\dots,g_m]$ $\in C^2(\Omega;\mathbb{R}^{n\times m})$ are unknown vector fields, and $u\in L_{\rm loc}^\infty([0,\infty),\mathbb{R}^m)$ is the control input.

To derive rigorous EDMD error bounds for $\nabla f(x)$ and $\nabla g(x)$, we differentiate system \eqref{e:control} with respect to time under constant control $\bar{u}=[\bar u_1,\cdots, \bar u_m]$, giving
\begin{eqnarray}\label{e:derivative-system}
\ddot{x} = \nabla f(x) \dot{x} + \sum_{i=1}^m \nabla g_i(x) \dot{x}\, \bar{u}_i.
\end{eqnarray}
Introducing the augmented state $y=\dot{x}$ and auxiliary control $v$, we obtain the following control-affine augmented system:
\begin{subequations}\label{e:augment-system}
\begin{eqnarray}
\dot{x} &=& f(x) + g(x) u, \label{e:augment-system-1}\\
\dot{y} &=& \nabla f(x) y + \sum_{i=1}^m \nabla g_i(x) y\, v_i.\label{e:augment-system-2}
\end{eqnarray}
\end{subequations}

In compact form, the augmented system reads
\begin{eqnarray}\label{e:augment-system2}
\dot{\xi} = F(\xi) + G(\xi) U,
\end{eqnarray}
with $\xi=[x^\top,y^\top]^\top$, $U=[u^\top,v^\top]^\top$, and
\begin{eqnarray}\label{e:aug-field}
F(\xi) = \begin{bmatrix} f(x) \\ \nabla f(x) y \end{bmatrix}, \quad G(\xi) = \begin{bmatrix} g(x) & 0 \\ 0 & \nabla g(x) y \end{bmatrix},
\end{eqnarray}
where $\nabla g(x) y \triangleq (\nabla g_1(x)y,\dots, \nabla g_m(x)y)$.

\subsection{Proper observable selection}
We apply the Koopman-based EDMD method to approximate the augmented system \eqref{e:augment-system2}. To exploit the linear dependence of \eqref{e:augment-system-2} on $y$, we first construct observables of the form
$
\Psi(x,y)
= \left[1, x^\top, \phi_1(x), \dots, \phi_N(x), y^\top, \psi_1(x,y), \dots, \psi_K(x,y)\right],
$
and establish the following local existence result.

\begin{prop}[Local existence of linear-in-$y$ observables]\label{p:linear-obs}
For any $\hat{x} \in \Omega$ and $\hat{v} \in \tilde{\mathbb{U}}$, there exist neighborhoods $\mathbb{B}_\rho(\hat{x}) \subset \Omega$, $\mathbb{B}_\rho(\hat{v}) \subset \tilde{\mathbb{U}}$, and observables $\Upsilon(x,y) = \{\psi_1, \dots, \psi_K\}$ on $\mathbb{B}_\rho(\hat{x}) \times \mathbb{R}^n$ such that:
    (1) Each $\psi_j(x,y)$ is linear in $y$;
    (2) $\mathbb{V}_y = \operatorname{span}\{\psi_1, \dots, \psi_K\}$ is invariant under the flow of \eqref{e:augment-system-2} for all $x \in \mathbb{B}_\rho(\hat{x})$, $v \in \mathbb{B}_\rho(\hat{v})$.
\end{prop}

\begin{proof}
We prove the result by following steps.

1. Treat system \eqref{e:augment-system-2} as a linear ODE in $y$ with parameters $x,v$. For any $(\hat{x},\hat{v})\in\Omega\times\tilde{\mathbb{U}}$, write
$
\dot{y} = A(\hat{x}, \hat{v}) y$, with $ A(\hat{x}, \hat{v}) = \nabla f(\hat{x}) + \sum_{i=1}^m \nabla g_i(\hat{x}) \hat{v}_i.
$
Let $\lambda_j(\hat{x},\hat{v})$ be eigenvalues of $A(\hat{x},\hat{v})$, $V_j(\hat{x},\hat{v})$ the generalized eigenspace of $A^*(\hat{x},\hat{v})$ associated with $\bar{\lambda}_j(\hat{x},\hat{v})$, and $\{w_j^k(\hat{x},\hat{v})\}_{k=1}^{n_j}$ an orthonormal basis of $V_j(\hat{x},\hat{v})$.
Define observables
$
\psi_j^{k_j}(y; \hat{x},\hat{v}) = \left\langle y, w_j^{k_j}(\hat{x},\hat{v}) \right\rangle, k_j = 1, \dots, n_j.
$
A direct computation shows that
\begin{align}
&\dot{\psi}_j^{k_j}(y; \hat{x},\hat{v})
= \left\langle \dot{y}, w_j^{k_j}(\hat{x},\hat{v}) \right\rangle
= \left\langle A(\hat{x},\hat{v}) y, w_j^{k_j}(\hat{x},\hat{v}) \right\rangle\notag\\
&= \left\langle y, A^*(\hat{x},\hat{v}) w_j^{k_j}(\hat{x},\hat{v}) \right\rangle \notag \\
&= \left\langle y, \begin{bmatrix} w_j^1(\hat{x},\hat{v}) & \cdots & w_j^{n_j}(\hat{x},\hat{v}) \end{bmatrix} B_j^*(\hat{x},\hat{v}) e_{k_j} \right\rangle \notag \\
&= \begin{bmatrix} \psi_j^1(y; \hat{x},\hat{v}) & \cdots & \psi_j^{n_j}(y; \hat{x},\hat{v}) \end{bmatrix} B_j^*(\hat{x},\hat{v}) e_{k_j},
\end{align}
where $e_{k_j}$ denotes the $k_j$-th standard basis vector in $\mathbb{R}^{n_j}$. Hence
$\operatorname{span}\left\{ \psi_j^{k_j}(y; \hat{x},\hat{v}) \mid k_j = 1, \dots, n_j \right\}$
is invariant under the flow of \eqref{e:augment-system-2} for fixed $(\hat{x},\hat{v})$.

2.
Let $P_j(\hat{x},\hat{v})$ denote the orthogonal projection onto $V_j(\hat{x},\hat{v})$. By the integral representation of spectral projections \cite[Ch. II.1]{Kato1995Perturbation},
$
P_j(\hat{x},\hat{v}) = -\frac{1}{2\pi i} \int_{\Gamma_j} \left( A^*(\hat{x},\hat{v}) - \zeta I \right)^{-1} d\zeta,
$
where $\Gamma_j\subset\mathbb{C}$ is a small counterclockwise circle centered at $\lambda_j(\hat{x},\hat{v})$. Differentiable dependence of $A^*$ on $(\hat{x},\hat{v})$ implies differentiable dependence of $P_j$ on $(\hat{x},\hat{v})$ \cite[Thm. II.5.4]{Kato1995Perturbation}.

3.
To extend the construction to a neighborhood of $(\hat{x},\hat{v})$, let $e_l$ be the $l$-th standard basis vector of $\mathbb{R}^n$. For each $j$, the set $\mathcal{V}_j(x,v) = \{P_j(x,v)e_1,\dots,P_j(x,v)e_n\}$.
Each vector $P_j(\hat{x},\hat{v}) e_l$ depends differentiably on $(\hat{x},\hat{v})$. Since $\operatorname{span}\mathcal{V}_j(\hat{x},\hat{v}) = V_j(\hat{x},\hat{v})$, the vectors in $\mathcal{V}_j(\hat{x},\hat{v})$ have rank $n_j$. Therefore, there exists a small neighborhood $\mathcal{N}(\hat{x},\hat{v})$ of $(\hat{x},\hat{v})$ and indices $k_1^j, \dots, k_{n_j}^j$ such that $\left\{ P_j(x,v) e_{k_1^j}, \dots, P_j(x,v) e_{k_{n_j}^j} \right\}$ forms a basis for $V_j(x,v)$ for all $(x,v) \in \mathcal{N}(\hat{x},\hat{v})$.

Gram--Schmidt orthonormalization yields a smoothly varying orthonormal basis $\{w_j^k(x,v)\}_{k=1}^{n_j}$ for $V_j(x,v)$ on $\mathcal{N}(\hat{x},\hat{v})$. Combining bases across all generalized eigenspaces gives a smooth orthonormal basis for $\mathbb{R}^n$.
Defining
$
\psi_j^{k_j}(y; x,v) = \left\langle y, w_j^{k_j}(x,v) \right\rangle,  k_j \in [1: n_j], j \in [1: r],\notag
$
with $r$ the number of distinct eigenvalues of $A(x,v)$, we obtain $n$ linearly independent observables satisfying both conditions of the proposition on $\mathcal{N}(\hat{x},\hat{v})$. This completes the proof.
\end{proof}

Proposition \ref{p:linear-obs} yields locally $y$-linear invariant observables for the $y$-subsystem. Global dictionary invariance over $\Omega\times \tilde{\mathbb U}$ demands topology dependent compatibility conditions outside our scope; we herein assume the aforementioned observables exist on $\Omega\times \tilde{\mathbb U}$.

Define the augmented dictionary space
\begin{equation}\label{e:observables}
\tilde{\mathbb{V}} = \operatorname{span}\left\{ 1, x^\top, \phi_{n+1}(x), \dots, \phi_N(x),\right.
\left. y^\top, \psi_{n+1}(x,y), \dots, \psi_K(x,y) \right\},
\end{equation}
where
$
\psi_k(x,y) = \sum_{j=1}^n a_{kj}(x) y_j:=a_k(x)^\top y,
$
with smooth functions $a_k(x)=(a_{k1}(x),\cdots, a_{kn}(x))^{\top}$ defined on $\Omega$. We set $\phi_i(x)=x_i$, $\psi_j(x,y)=y_j$ for $i,j=1,\dots,n$.
Let $L^2_Y(\Omega)$ denote the function space of the form $\phi(x)y_i^\iota$ ($\iota=0,1$) over the symmetric domain $Y_a=[-a,a]^n$, so $L^2_Y(\Omega)=L^2(\Omega)\oplus \bigoplus_{i=1}^n L^2(\Omega)y_i$. Decompose the dictionary as
\begin{align}\label{e:vx-vy}
\mathbb{V}_x& = \operatorname{span}\{1, x^\top, \phi_1(x), \dots, \phi_N(x)\},\\
\mathbb{V}_y &= \operatorname{span}\{\psi_1(x,y), \dots, \psi_K(x,y)\},\notag
\end{align}
hence $\tilde{\mathbb{V}} = \mathbb{V}_x \oplus \mathbb{V}_y$. The $0$-symmetry of $Y_a$ implies the following orthogonality property.

\begin{lemma}
$\mathbb{V}_x$ and $\mathbb{V}_y$ are orthogonal in $L^2_Y(\Omega)$.
\end{lemma}

\begin{proof}
For any $\phi(x)\in \mathbb{V}_x$ and $\psi(x,y)\in \mathbb{V}_y$, linearity of $\psi$ in $y$ and symmetry of $Y_a$ yield
$
\langle \phi,\psi\rangle_{L^2_Y(\Omega)}=\int_{\Omega}\phi(x)\left(\int_{Y_a}\psi(x,y)dy\right)dx=0,
$
since the integral of a linear function over a symmetric domain vanishes.
\end{proof}

\subsection{EDMD error bounds: invariant dictionary space}
\begin{lemma}\label{l:v-x-invariant}
Under Assumption 1, for the augmented dictionary space $\tilde{\mathbb{V}}$, the subspace $\mathbb{V}_x$ is invariant under the flow of system \eqref{e:augment-system-1}--\eqref{e:augment-system-2}.
\end{lemma}

\begin{proof}
Take any constant control $U_0=(u_0,v_0)\in\mathbb{U}\times\tilde{\mathbb{U}}$. Since $\phi_k(x)$ depends only on $x$, its augmented gradient is $\nabla_\xi \phi_k = [\nabla_x \phi_k^\top, 0]^\top$. The Lie derivative along $F(\xi)+G(\xi)U_0$ satisfies
$
\mathcal{L}^{(u_0, v_0)} \phi_k
= \left(f(x)+g(x)u_0\right)^\top\nabla_x\phi_k(x),
$
which is independent of $y$. By the invariance of $\tilde{\mathbb{V}}$, this Lie derivative lies in $\tilde{\mathbb{V}}$; as it contains no $y$-dependent terms, it must belong to $\mathbb{V}_x$.
\end{proof}

\begin{lemma}\label{l:v-y-invariant}
Under Assumption 1, for the augmented dictionary space $\tilde{\mathbb{V}}$, the subspace $\mathbb{V}_y$ is invariant under the flow of system \eqref{e:augment-system-1}--\eqref{e:augment-system-2}.
\end{lemma}

\begin{proof}
For any constant control $\bar{U}=(\bar{u},\bar{v})\in\mathbb{U}\times\tilde{\mathbb{U}}$, the Lie derivative of $\psi_k(x,y)$ along $F(\xi)+G(\xi)\bar{U}$ reads
\begin{equation}\label{e:LU-lie}
\mathcal{L}^{\bar{U}} \psi_k= \left(f(x)+g(x)\bar{u}\right)^\top\nabla_x\psi_k 
+ \left[\nabla f(x)y+\sum_{i=1}^m\nabla g_i(x)y\,\bar{v}_i\right]^\top\nabla_y\psi_k.
\end{equation}
Since $\psi_k(x,y)=a_k(x)^\top y$ is linear in $y$, \eqref{e:LU-lie} yields
$
\mathcal{L}^{\bar{U}} \psi_k = \gamma_k^{\bar{U}}(x)^\top y
$
for some smooth vector-valued function $\gamma_k^{\bar{U}}(x)$, so the Lie derivative remains linear in $y$.

By the invariance of $\tilde{\mathbb{V}}$, $\mathcal{L}^{\bar{U}} \psi_k$ can be expressed as a linear combination of all basis functions in $\tilde{\mathbb{V}}$. Matching the linear-in-$y$ structure shows that coefficients of all $x$-only basis terms must be zero.
This establishes the invariance of $\mathbb{V}_y$.
\end{proof}

\begin{theorem}[EDMD error bounds for invariant dictionary space]\label{t:bounds-aug}
Let Assumption 1 hold for $\tilde{\mathbb{V}}$, and let the training data samples be i.i.d.. For any $\delta\in(0,1)$ and $\eta>0$, there exists $d_0\in\mathbb{N}$ such that for all $d\geq d_0$, the augmented system \eqref{e:augment-system} admits the data-driven representation
\begin{align}\label{e:equivalent-system}
\dot{\xi} = \left( F^\eta(\xi) + \alpha(\xi) \right) + \left( G^\eta(\xi) + \beta(\xi) \right) U,
\end{align}
with block-structured errors
\begin{align}\label{e:alpha-beta}
\alpha(\xi) = \begin{bmatrix} \alpha_1(x) \\ \alpha_2(x)^\top y \end{bmatrix}, \quad \beta(\xi) = \begin{bmatrix} \beta_1(x) & 0 \\ 0 & \beta_2(x)^\top y \end{bmatrix},
\end{align}
satisfying
\begin{align}\label{e:f-g-error-3}
\|\alpha_1(x)\| \leq \eta \|x\|, \quad \|\alpha_2(x)\| \leq \eta, \quad \|\beta_i(x)\| \leq \eta,
\end{align}
for all $x\in\Omega$ and $i=1,2$, with probability at least $1-\delta$.
\end{theorem}

\begin{remark}
Compared with Proposition \ref{p:app}, here the residual terms explicitly preserve the linear-in-$y$ structure of the augmented dynamics.
\end{remark}

\begin{proof}
The proof builds on Proposition \ref{p:app} and the invariance properties of $\tilde{\mathbb{V}}$, proceeding in three steps.

1. By Lemma \ref{l:v-x-invariant}, the Lie derivative of $x$-coordinate observables $\phi_k(x)=x_k$ lies in $\mathbb{V}_x$. By the form of Lie operators of the augmented system, the EDMD approximation $f^\eta(x)$ of the drift yields the error
$
\alpha_1(x) = f(x)-f^\eta(x) = P_x\left[ \left(\mathcal{L}^{(0,0)}-\mathcal{L}_d^{(0,0)}\right)\Phi(\xi) \right],
$
which depends only on $x$. Using $\|P_x\|\leq1$ and $\|\Phi(\xi)\|\leq L_\Phi\|x\|$, we obtain
$
\|\alpha_1(x)\| \leq \left\|\mathcal{L}^{(0,0)}-\mathcal{L}_d^{(0,0)}\right\| L_\Phi \|x\| \leq c_r L_\Phi \|x\|.
$

For the $y$-component, Lemma \ref{l:v-y-invariant} ensures the Lie derivative of $y$-linear observables remains linear in $y$. The EDMD error takes the form $\alpha_2(x)^\top y$, where
$
\alpha_2(x)^\top y = P_y\left[ \left(\mathcal{L}^{(0,0)}-\mathcal{L}_d^{(0,0)}\right)\Upsilon(\xi) \right],
$
with $\Upsilon(\xi)=(\psi_1(\xi),\cdots,\psi_K(\xi))^\top$.
Taking the supremum over $\|y\|=1$ gives
$
\|\alpha_2(x)\| \leq L_{\Upsilon_y} \left\|\mathcal{L}^{(0,0)}-\mathcal{L}_d^{(0,0)}\right\| \leq L_{\Upsilon_y} c_r,
$
with $L_{\Upsilon_y}$ the maximum coefficient magnitude of the $y$-linear basis. Choosing $d_0$ large enough so that $c_r L_\Phi \leq \eta$ and $c_r L_{\Upsilon_y}\leq\eta$ yields the drift error bounds.

2.
Analogously, $G^\eta(\xi)$ inherits the block-diagonal structure of $G(\xi)$, with error
$
\beta(\xi) = \begin{bmatrix} \beta_1(x) & 0 \\ 0 & \beta_2(x)^\top y \end{bmatrix},
$
where $\beta_{1j}(x) = g_j(x) - G^\eta_{1j}(\xi)$ and $\beta_{2j}(x)^\top y = \nabla g_j(x)^\top y - G^\eta_{2j}(\xi)$ for $j\in [1:m]$.
For each $\beta_{1j}(x)$, using $\|P_x\| \leq 1$, the general EDMD error bound \eqref{e:error-data}, and Lemmas \ref{l:v-x-invariant} and \ref{l:v-y-invariant}, we obtain
\begin{align}\label{e:beta1j}
\|\beta_{1j}(x)\|
&= \left\| P_x \left[ \left( \mathcal{L}^{(E_j,0)} - \mathcal{L}_d^{(E_j,0)} \right) \Phi(\xi) \right] \right\| \notag \\
&\leq C_\Phi \left\| \mathcal{L}^{(E_j,0)} - \mathcal{L}_d^{(E_j,0)} \right\| \leq C_\Phi c_r,
\end{align}
where $C_\Phi = \sup_{\xi\in\Omega\times Y_a} \|\Phi(\xi)\|$.
For each $\beta_{2j}(x)$, using $\|P_y\| \leq 1$, we have for all $x\in\Omega$ and $y\in Y_a$
\begin{align}\label{e:beta2j}
\left\| \beta_{2j}(x)^\top y \right\|
&= \left\| P_y \left[ \left( \mathcal{L}^{(0,E_j)} - \mathcal{L}_d^{(0,E_j)} \right) \Psi(\xi) \right] \right\|\\
&\leq \left\| \mathcal{L}^{(0,E_j)} - \mathcal{L}_d^{(0,E_j)} \right\| \|\Upsilon(\xi)\| \leq c_r L_{\Upsilon} \|y\|.\notag
\end{align}
Sufficiently large $d_0$ ensures both terms are bounded by $\eta$.

3.
Combining all estimates, there exists $d_0\in\mathbb{N}$ such that for all $d\geq d_0$, bounds \eqref{e:f-g-error-3} hold with probability at least $1-\delta$.
\end{proof}

\subsection{EDMD error bounds: Non-Invariant Dictionary Spaces}\label{s:Error-bounds}

\begin{theorem}[EDMD error bounds for non-invariant dictionary space]\label{t:bounds-aug-general}
Let $\tilde{\mathbb{V}}$ be given by \eqref{e:observables} with subspaces $\mathbb{V}_x,\mathbb{V}_y$ from \eqref{e:vx-vy}, and let training samples be i.i.d.. For any $\delta\in(0,1)$ and $\eta>0$, there exists $d_0\in\mathbb{N}$ such that for all $d\geq d_0$, the augmented system \eqref{e:augment-system} admits the data-driven representation
\begin{align}\label{e:equivalent-system-general}
\dot{\xi} = \big( F^\eta(\xi) + \alpha(\xi) \big) + \big( G^\eta(\xi) + \beta(\xi) \big) U,
\end{align}
with block-structured errors as \eqref{e:alpha-beta}
satisfying
\begin{align}\label{e:f-g-error-general}
&\|\alpha_1(x)\| \leq \big( C_{\mathcal{L}} \mathcal{I}_{\max} L_\Phi + \eta \big) \|x\|, \notag\\
&\|\alpha_2(x)\| \leq C_{\mathcal{L}} \mathcal{I}_{\max} L_{\Psi_y} + \eta, \notag\\
&\|\beta_i(x)\| \leq C_{\mathcal{L}} \mathcal{I}_{\max} C_\Psi + \eta, \quad i=1,2,
\end{align}
for all $x\in\Omega$, with probability at least $1-\delta$. Here $C_{\mathcal{L}} = \sup_{U} \| \mathcal{L}^U |_{\tilde{\mathbb{V}}} \|$, $\mathcal{I}_{\max} = \sup_{U} \mathcal{I}_{\mathcal{L}^U}(\tilde{\mathbb{V}})$, $L_\Phi, L_{\Psi_y}$ are Lipschitz constants, and $C_\Psi=\sup_{x\in\bar\Omega} |\Psi(x)|$.
\end{theorem}

\begin{proof}
We extend the invariant-case analysis by incorporating projection errors from dictionary non-invariance.

1.
For any constant control $U$, the total EDMD error decomposes as \eqref{e:non-inv}:
\begin{align}\label{e:error-decomp-general}
\| \mathcal{L}^U |_{\tilde{\mathbb{V}}} - \mathcal{L}_d^U \| &\leq \| \mathcal{L}^U |_{\tilde{\mathbb{V}}} - P_{\tilde{\mathbb{V}}} \mathcal{L}^U |_{\tilde{\mathbb{V}}} \| + \| P_{\tilde{\mathbb{V}}} \mathcal{L}^U |_{\tilde{\mathbb{V}}} - \mathcal{L}_d^U \|.\notag\\
&\triangleq  E_{\text{proj}}(U)+E_{\text{stat}}(U)
\end{align}
where the statistical error satisfies $E_{\text{stat}}(U)\leq c_r$ with probability $1-\delta$ (Theorem \ref{t:app-l-ld}), and the projection error is bounded by
$E_{\text{proj}}(U) \leq \| \mathcal{L}^U |_{\tilde{\mathbb{V}}} \| \cdot \mathcal{I}_{\mathcal{L}^U}(\tilde{\mathbb{V}}) \leq C_{\mathcal{L}} \mathcal{I}_{\max}$.

2.
For any $\phi\in\mathbb{V}_x$, $\mathcal{L}^U\phi(x)$ depends only on $x$. By the $0$-symmetry of $Y_a$ and linearity of $\psi_k$ in $y$, the inner product $\langle\mathcal{L}^U\phi,\psi_k\rangle_{L^2(\Omega\times Y_a)}$ vanishes, so $P_{\tilde{\mathbb{V}}}\mathcal{L}^U\phi\in\mathbb{V}_x$. Hence the projection error $\mathcal{L}^U\phi - P_{\tilde{\mathbb{V}}}\mathcal{L}^U\phi$ is $x$-only.
For any $\psi_k\in\mathbb{V}_y$, $\mathcal{L}^U\psi_k = \gamma_k^U(x)^\top y$ remains linear in $y$. Similarly, $P_{\tilde{\mathbb{V}}}\mathcal{L}^U\psi_k\in\mathbb{V}_y$, so the projection error is also linear in $y$.
Thus the total error preserves the block structure as \eqref{e:alpha-beta}: errors on $\mathbb{V}_x$ are $x$-dependent, and errors on $\mathbb{V}_y$ are linear in $y$.

3. Following the same procedure as in Theorem \ref{t:bounds-aug},
for the $x$-component of drift error:
$
\alpha_1(x) = P_x \big[ \big( \mathcal{L}^{(0,0)} - \mathcal{L}_d^{(0,0)} \big) \Phi(\xi) \big].
$
Using \eqref{e:error-decomp-general}, $\|\Phi(\xi)\|\leq L_\Phi\|x\|$ and the total error bound:
$
\|\alpha_1(x)\| \leq \big( C_{\mathcal{L}} \mathcal{I}_{\max} + c_r \big) L_\Phi \|x\|.
$
Choosing $d_0$ such that $c_r L_\Phi\leq\eta$ yields the first bound.
For the $y$-component, $\alpha_2(x)^\top y = P_y \big[ \big( \mathcal{L}^{(0,0)} - \mathcal{L}_d^{(0,0)} \big) \Psi(\xi) \big]$. Taking $\sup_{\|y\|=1}$ gives
$
\|\alpha_2(x)\| \leq \big( C_{\mathcal{L}} \mathcal{I}_{\max} + c_r \big) L_{\Psi_y}.
$
Choosing $c_r L_{\Psi_y}\leq\eta$ for sufficiently large $d_0$ yields the second bounds in \eqref{e:f-g-error-general}.

4.  Similar to \eqref{e:beta1j} and \eqref{e:beta2j},
for the $x$-component $\beta_{1j}(x)$ corresponding to control $(E_j,0)$:
$
\|\beta_{1j}(x)\| \leq \| \mathcal{L}^{(E_j,0)} - \mathcal{L}_d^{(E_j,0)} \| C_\Psi \leq \big( C_{\mathcal{L}} \mathcal{I}_{\max} + c_r \big) C_\Psi,
$
and, for the $y$-component $\beta_{2j}(x)$ corresponding to $(0,E_j)$:
$
\|\beta_{2j}(x)\| \leq \big( C_{\mathcal{L}} \mathcal{I}_{\max} + c_r \big) L_{\Psi_y} \leq \big( C_{\mathcal{L}} \mathcal{I}_{\max} + c_r \big) C_\Psi.
$
Choosing $d_0$ such that $c_r C_\Psi\leq\eta$ completes the proof.
\end{proof}

\begin{remark}
When $\tilde{\mathbb{V}}$ is exactly invariant ($\mathcal{I}_{\max}=0$), Theorem \ref{t:bounds-aug-general} reduces to Theorem \ref{t:bounds-aug}. In practice, the Approximated-SSD algorithm (Section 8.1) selects a subspace $\tilde{\mathbb{V}}'\subset\tilde{\mathbb{V}}$ with sufficiently small $\mathcal{I}_{\max}$, rendering the projection error sufficiently small for applications.
\end{remark}

\textbf{Assumption 2.}
For dictionary $\tilde{\mathbb{V}}$, there exists a subspace $\tilde{\mathbb{V}}'\subset\tilde{\mathbb{V}}$ such that:
(1) $x^\top$ and $y^\top$ are contained in $\tilde{\mathbb{V}}'$;
(2) $C_{\mathcal{L}} \mathcal{I}_{\max} L_\Phi\le \eta$, $C_{\mathcal{L}} \mathcal{I}_{\max} L_{\Psi_y}\le \eta$, $C_{\mathcal{L}} \mathcal{I}_{\max} C_\Psi\le \eta$.

Let $\theta^\eta(x)$ and $\chi^\eta(x)$ denote the EDMD approximations of $\nabla f(x)$ and $\nabla g(x)$, respectively. Then Theorems \ref{t:bounds-aug} and \ref{t:bounds-aug-general} yield that
\begin{align}\label{e:theta-chi}
\nabla f(x) - \theta^\eta(x) = \alpha_2(x), \quad \nabla g(x) - \chi^\eta(x) = \beta_2(x).
\end{align}

For $(u,v)\in\mathbb{U}\times \tilde{\mathbb{U}}$, along any trajectory $(x(t), y(t))$,
$
\frac{d}{dt} \Psi(x(t), y(t)) = \mathcal{L}^{(u,v)} \Psi(x(t), y(t)).
$
By \eqref{e:lie-op-Lu},
\begin{equation}\label{e:lie-op-Lu-v}
\mathcal{L}^{(u,v)} = 
\mathcal{L}^{(0,0)} + \sum_{i=1}^{m} u_i \left[ \mathcal{L}^{(E_i,0)} - \mathcal{L}^{(0,0)} \right]
+ \sum_{i=1}^{m} v_i \left[ \mathcal{L}^{(0,E_i)} - \mathcal{L}^{(0,0)} \right].
\end{equation}

\section{Stable Manifolds of HJB Equations}\label{s:stable}

We briefly review the stable manifold approach for HJB equations in nonlinear optimal control.
Consider the control-affine system \eqref{e:control} with infinite-horizon cost
\begin{equation}\label{e:cost}
J(x, u) = \int_0^\infty \left(q(x) + \frac{1}{2} u^\top W u\right) dt, \quad W>0,
\end{equation}
where $q(x)\ge0$ is smooth. The associated HJB equation reads
\begin{equation}\label{e:HJB}
\nabla V(x)^\top f(x) - \frac{1}{2} \nabla V(x)^\top R(x) \nabla V(x) + q(x) = 0,
\end{equation}
with $R(x) = g(x) W^{-1} g(x)^\top$, and the Hamiltonian function
$H(x, p) = p^\top f(x) - \frac{1}{2} p^\top R(x) p + q(x)$.

The optimal feedback law is $u(x) = -W^{-1} g(x)^\top \nabla V(x)$, yielding the closed-loop system
\begin{align}\label{e:system-closed}
\dot{x} = f(x) - R(x) \nabla V(x).
\end{align}
A solution $V$ is \emph{stabilizing} if $\nabla V(0)=0$ and the origin is asymptotically stable for \eqref{e:system-closed}.
The characteristic Hamiltonian system of \eqref{e:HJB} is
\begin{align}\label{e:Hamiltonian-flow}
\begin{cases}
\dot{x} = f(x) - R(x) p, \\
\dot{p} = -\left(\frac{\partial f}{\partial x}\right)^\top p + \frac{1}{2} \frac{\partial}{\partial x}\left(p^\top R(x) p\right) - \left(\frac{\partial q}{\partial x}\right)^\top.
\end{cases}
\end{align}
The gradient graph $\Lambda_V = \{(x,p)\mid p=\nabla V(x)\}$ is an invariant manifold of \eqref{e:Hamiltonian-flow}. Linearizing at the origin gives the Hamiltonian matrix
\begin{align}\label{e:Ham-mat}
\operatorname{Ham} = \begin{bmatrix} A & -R(0) \\ -Q & -A^\top \end{bmatrix},
\end{align}
where $A = Df(0)$, $Q = D^2 q(0)$.

\textbf{Assumption 3.}
\begin{enumerate}
  \item[$(\mathbf{A_{31}})$] $f,g,q$ are $C^\infty$ on $\bar{\Omega}$, with $f(x)=Ax+O(\|x\|^2)$, $q(x)=\frac{1}{2}x^\top Qx+O(\|x\|^3)$.
  \item[$(\mathbf{A_{32}})$] ${\rm Ham}$ is hyperbolic, and its stable eigenspace $E_-$ satisfies $E_- \oplus {\rm Im}[0,\;I_n]^\top = \mathbb{R}^{2n}$.
\end{enumerate}

\begin{theorem}[Local stable manifold theorem]\label{t:stable-manifolds}
Under Assumption 3, the stable manifold of \eqref{e:Hamiltonian-flow} through the origin is an $n$-dimensional smooth submanifold, locally coinciding with $\Lambda_V$. For small $\|x\|$,
$
\|p(x) - Px\| \le k\|x\|^2,
$
where $P = D^2 V(0)$ solves the algebraic Riccati equation
\begin{equation}\label{e:riccati}
P A + A^\top P - P R(0) P + Q = 0.
\end{equation}
\end{theorem}

\begin{remark}\label{r:regularity}
We assume $V\in C^2(\Omega)$ and the graph representation $\Lambda_V$ holds semiglobally on $\Omega$.
\end{remark}

Following \cite{sakamoto2008analytical}, we introduce the coordinate transformation
\begin{equation}\label{e:t-trans}
T = \begin{bmatrix} I_n & S \\ P & P S + I_n \end{bmatrix}, \quad
\begin{bmatrix} \bar{x} \\ \bar{p} \end{bmatrix} = T^{-1} \begin{bmatrix} x \\ p \end{bmatrix},
\end{equation}
where $S$ solves the Lyapunov equation
$
(A - R(0) P) S + S (A - R(0) P)^\top = R(0).
$
Then system \eqref{e:Hamiltonian-flow} becomes
\begin{align}\label{e:Hamiltonian-flow-2}
\dot{\bar{x}} = \Theta \bar{x} + N_s(\bar{x},\bar{p}), \quad
\dot{\bar{p}} = -\Theta^\top \bar{p} + N_u(\bar{x},\bar{p}),
\end{align}
where $\Theta = A - R(0)P$, $N_s$ and $N_u$ denote the nonlinear remainder terms obtained after extracting the linear part of \eqref{e:Hamiltonian-flow}. We solve this transformed system under the boundary conditions
\begin{align}\label{e:bvp-n}
\bar{x}(0) = \bar{x}_0, \quad \bar{p}(+\infty) = 0.
\end{align}

\begin{prop}\label{p:bvp-existence}
\cite[Theorem 5]{sakamoto2008analytical}
For sufficiently small $\bar{x}_0$, the BVP \eqref{e:Hamiltonian-flow-2}--\eqref{e:bvp-n} has a unique solution, lying on the local stable manifold of the origin.
\end{prop}

\section{Data-Driven Approximations of Stable Manifold and Optimal Control}\label{s:edmd-sm}

Based on the augmented-system error bounds in Theorem \ref{t:bounds-aug-general}, we construct a data-driven approximation of the characteristic Hamiltonian system, derive rigorous error estimates for the stable manifold, and obtain the corresponding nearly optimal control.

\subsection{Approximate Characteristic Hamiltonian System}
Define the data-driven approximation of \eqref{e:Hamiltonian-flow} as
\begin{align}\label{e:Ham-flow-app}
\begin{cases}
\dot{x} = f^\eta(x) - R^\eta(x) p, \\
\dot{p} = -\theta^\eta(x)^\top p + \frac{1}{2} \tilde{R}_d(x, p) - \nabla q(x)^\top,
\end{cases}
\end{align}
where $R^\eta(x) = g^\eta(x) W^{-1} g^\eta(x)^\top$, $\theta^\eta(x), \chi^\eta(x)$ are EDMD approximations of $\nabla f(x), \nabla g(x)$ (cf. \eqref{e:theta-chi}), and
\begin{align}\label{e:def-R-d}
\tilde{R}_d(x, p) = p^\top \left( \chi^\eta(x) W^{-1} g^\eta(x)^\top + g^\eta(x) W^{-1} \chi^\eta(x)^\top \right) p
\end{align}
approximates $\frac{\partial}{\partial x}\left(p^\top R(x) p\right)$.
For sufficiently small $\|x_0\|$, the two-point BVP for system \eqref{e:Ham-flow-app} with boundary conditions
\begin{align}\label{e:two-point-bvp}
x(0) = x_0, \quad p(+\infty) = 0
\end{align}
admits a unique local solution near the origin.

\begin{remark}
System \eqref{e:Ham-flow-app} is generally not an exact Hamiltonian system. It reduces to the exact Hamiltonian system for $f^\eta, g^\eta$ if $\theta^\eta(x) = \nabla f^\eta(x)$ and $\chi^\eta(x) = \nabla g^\eta(x)$. In practice, EDMD approximates Lie derivatives rather than direct gradients, so this consistency rarely holds. Our method bypasses this requirement by constructing the control directly from the stable manifold of \eqref{e:Ham-flow-app}.
\end{remark}

\begin{lemma}\label{l:app-condition}
Under Assumptions 2 and 3, there exists $\eta_0>0$ such that for all $\eta\in[0,\eta_0)$, the EDMD approximations $f^\eta,g^\eta$ satisfy Condition $(\mathbf{A_{31}})$, and system \eqref{e:Ham-flow-app} satisfies $(\mathbf{A_{32}})$.
\end{lemma}

\begin{proof}
By construction, $f^\eta,g^\eta$ are $C^\infty$ on $\bar{\Omega}$. By Theorem \ref{t:bounds-aug-general}, $f^\eta(x)=A^\eta x+O(\|x\|^2)$ with $A^\eta=\theta^\eta(0)$, verifying $(\mathbf{A_{31}})$.
The Jacobian of \eqref{e:Ham-flow-app} at the origin is
\begin{align}\label{e:Ham-mat-app}
\operatorname{Ham}^\eta = \begin{bmatrix}
A^\eta & -R^\eta(0) \\
-Q & -\theta^\eta(0)^\top
\end{bmatrix},
\end{align}
a perturbation of $\operatorname{Ham}$. By error bounds \eqref{e:theta-chi} and Theorem \ref{t:bounds-aug-general}, $\|A^\eta-A\|\le\eta$, $\|\theta^\eta(0)-A\|\le\eta$, and $\|R^\eta(0)-R(0)\|\le C_{W,g}\eta$ for some constant $C_{W,g}$ depending only on $W$ and $g$. By classical operator perturbation theory \cite[Chapter 2.1]{Kato1995Perturbation}, there exists small $\eta_0>0$ such that for all $\eta\in [0,\eta_0)$ hyperbolicity and the complementary condition are preserved.
\end{proof}

Let $\operatorname{Proj}_-^\eta$ (resp. $\operatorname{Proj}_-^0$) denote the spectral projection onto the stable generalized eigenspace of $\operatorname{Ham}^\eta$ (resp. $\operatorname{Ham}$). By the complementary condition, there exist matrices $P^\eta,P$ such that
\begin{align}\label{e:stable-eigen}
\operatorname{Proj}_-^0 \mathbb{R}^{2n} = \operatorname{span}\begin{bmatrix} I_n \\ P \end{bmatrix}, \quad
\operatorname{Proj}_-^\eta \mathbb{R}^{2n} = \operatorname{span}\begin{bmatrix} I_n \\ P^\eta \end{bmatrix},
\end{align}
where $P>0$ is the unique stabilizing solution to the algebraic Riccati equation \eqref{e:riccati}. Continuity of spectral projections (\cite[Chapter 2]{Kato1995Perturbation}) implies:

\begin{lemma}\label{l:p-app}
For any $\varepsilon>0$, there exists $\eta_0>0$ such that $\|P^\eta-P\|<\varepsilon$ for all $\eta\in[0,\eta_0)$.
\end{lemma}

\begin{remark}\label{r:ham-app}
Equation \eqref{e:stable-eigen} implies that there exists a unique stable matrix $\Theta^\eta$ satisfying
$
\operatorname{Ham}^\eta \begin{bmatrix} I_n \\ P^\eta \end{bmatrix} = \begin{bmatrix} I_n \\ P^\eta \end{bmatrix} \Theta^\eta.
$
This yields that $\Theta^\eta=A^\eta-R^\eta(0)P^\eta$ and the generalized algebraic Riccati equation
\begin{equation}\label{e:general-Riccati}
P^\eta A^\eta + \theta^\eta(0)^\top P^\eta - P^\eta R^\eta(0) P^\eta + Q = 0,
\end{equation}
which reduces to the standard ARE \eqref{e:riccati} when $\eta=0$.
\end{remark}

For brevity, let $F_H(x,p)=[H_1(x,p)^\top,H_2(x,p)^\top]^\top$ and $F_H^\eta(x,p)=[H_1^\eta(x,p)^\top,H_2^\eta(x,p)^\top]^\top$ denote the exact and approximate Hamiltonian vector fields of \eqref{e:Hamiltonian-flow} and \eqref{e:Ham-flow-app}.

\begin{lemma}[Error bound for EDMD-approximated vector fields]\label{l:F-error}
Let $\Gamma\subset\mathbb{R}^n$ be a bounded domain. Under Assumptions 2 and 3, the exact and approximate Hamiltonian vector fields satisfy
$F_H(x,p) = F_H^\eta(x,p) + \Lambda(x,p)$ with
\begin{align}\label{e:lambda}
\|\Lambda(x,p)\| \leq C_{f,g,W} \eta \|(x,p)\|,\quad \forall (x,p)\in\Omega\times\Gamma,
\end{align}
where $C_{f,g,W}>0$ depends only on $f,g,W$.
\end{lemma}

\begin{proof}
Let $\Lambda=[\Lambda_1^\top,\Lambda_2^\top]^\top$ and bound each component.

1.
Write $\Lambda_1 = [f(x)-f^\eta(x)] + [R^\eta(x)-R(x)]p \triangleq \alpha_1(x) + \Delta R(x)\,p$.
By Theorem \ref{t:bounds-aug-general}, $\|\alpha_1(x)\|\leq\eta\|x\|\leq\eta\|(x,p)\|$.
Expanding $\Delta R = g^\eta W^{-1}(g^\eta-g)^\top + (g^\eta-g)W^{-1}g^\top$ and using $\|g^\eta-g\|\leq\eta$, uniform boundedness of $g$ on compact $\bar\Omega$, and boundedness of $W^{-1}$, we obtain $\|\Delta R(x)\|\leq C_{g,W}\eta$. Hence
$
\|\Lambda_1\| \leq (1+C_{g,W})\eta\|(x,p)\|.
$

2. For $\Lambda_2$,
write
$\Lambda_2 = [\theta^\eta(x)-\nabla f(x)]^\top p + \frac{1}{2}\left[\frac{\partial}{\partial x}(p^\top R p)-\tilde{R}_d(x,p)\right]$.
By gradient error bounds and Assumption 2, $\|\theta^\eta-\nabla f\|\leq 2\eta$, so the first term is bounded by $2\eta\|p\|$.

For the second term, we have
$
\frac{\partial}{\partial x} \left( p^\top R(x) p \right) = p^\top \left( \nabla g(x) W^{-1} g(x)^\top + g(x) W^{-1} \nabla g(x)^\top \right) p.
$
From the definition of $\tilde{R}_d(x,p)$ in \eqref{e:def-R-d}, we decompose
\begin{align}
&\frac{\partial}{\partial x} \left[ p^\top R(x) p \right] - \tilde{R}_d(x,p)
= p^\top \left[ (\nabla g(x) - \chi^\eta(x)) W^{-1} g(x)^\top \right.\notag\\
&+\chi^\eta(x) W^{-1} (g(x) - g^\eta(x))^\top + (g(x) - g^\eta(x)) W^{-1} \nabla g(x)^\top\notag \\
&+\left. g^\eta(x) W^{-1} (\nabla g(x) - \chi^\eta(x))^\top \right] p. \notag
\end{align}

From Theorem \ref{t:bounds-aug-general} and Assumption 2, the gradient approximation error satisfies $\| \nabla g(x) - \chi^\eta(x) \| \leq 2\eta$. By Assumption 3, $\nabla g \in C^\infty(\bar{\Omega})$, and hence $\| \nabla g(x) \| \leq C_{\nabla g}$ for some constant $C_{\nabla g} > 0$ on the compact domain $\bar{\Omega}$. Hence
$
\left\| \frac{\partial}{\partial x} \left( p^\top R(x) p \right) - \tilde{R}_d(x,p) \right\| \leq  C_{g,W}' \eta \|(x,p) \|^2,
$
where $C_{g,W}'$ is constant depending only on $g,W$.  It follows that
$
\|\Lambda_2\| \leq \hat{C}_{g,W}'\,\eta\|(x,p)\|.
$
Combining both components gives \eqref{e:lambda} with $C_{f,g,W} = 1+C_{g,W}+\hat{C}_{g,W}'$.
\end{proof}

\subsection{Local Error Estimates for EDMD-Approximated Stable Manifolds}
Transformation \eqref{e:t-trans} decouples the linear part of the exact Hamiltonian system into stable and unstable blocks as in \eqref{e:Hamiltonian-flow-2}, with nonlinearities satisfying $N_s, N_u = O(\|(\bar{x},\bar{p})\|^2)$.

For the approximate system \eqref{e:Ham-flow-app}, the stable block is $\Theta^\eta = A^\eta - R^\eta(0)P^\eta$, which converges to $\Theta$ as $\eta\to0$. Since $\operatorname{Ham}^\eta$ is not Hamiltonian, the unstable block is not $-(\Theta^\eta)^\top$; instead, denote
$
\Xi^\eta = \left(\theta^\eta(0) - R^\eta(0)P^\eta\right)^\top.
$
Define $S^\eta$ as the unique solution to the asymmetric Lyapunov equation
\begin{align}\label{e:lypunov-app}
\Theta^\eta S^\eta + S^\eta \Xi^\eta = R^\eta(0).
\end{align}
As $\eta\to0$, $\Xi^\eta\to\Theta^\top$ and $S^\eta\to S$.

The generalized unstable eigenspace of $\operatorname{Ham}^\eta$ has basis $\begin{bmatrix} S^\eta \\ P^\eta S^\eta + I_n \end{bmatrix}$. Define the invertible transformation matrix
$
T^\eta = \begin{bmatrix} I_n & S^\eta \\ P^\eta & P^\eta S^\eta + I_n \end{bmatrix},
$
which converges to $T$ as $\eta\to0$, with $(T^\eta)^{-1}$ continuous in $\eta$.
Applying $(\bar{x}^\eta, \bar{p}^\eta) = (T^\eta)^{-1}(x,p)$ to \eqref{e:Ham-flow-app} yields
\begin{align}\label{e:hs-transform-app}
\begin{cases}
\dot{\bar{x}}^\eta = \Theta^\eta \bar{x}^\eta + N_s^\eta(\bar{x}^\eta, \bar{p}^\eta), \\
\dot{\bar{p}}^\eta = -\Xi^\eta \bar{p}^\eta + N_u^\eta(\bar{x}^\eta, \bar{p}^\eta),
\end{cases}
\quad \bar{x}^\eta(0,\xi) = \xi,
\end{align}
where $N_s^\eta, N_u^\eta = O(\|(\bar{x}^\eta, \bar{p}^\eta)\|^2)$ near the origin.

\subsubsection{Estimate for near the equilibrium}

\begin{lemma}\label{l:p-error-local}
For any $\varepsilon>0$, there exist $\sigma_0>0$ and $\eta_0>0$ such that for all $\eta\in[0,\eta_0)$,
\begin{equation}\label{e:p-eta-p-local}
\|p^\eta(x) - p(x)\| \le \varepsilon \|x\|,\quad \forall x\in B_{\sigma_0}(0).
\end{equation}
\end{lemma}

\begin{proof}
We use a perturbation argument based on Lyapunov--Perron operators.

1.
Let $\mathcal{E}^0$ be the Banach space of continuous functions $p:\mathbb{R}^n\to\mathbb{R}^n$ with $p(0)=0$ and finite norm $\|p\|_\mathcal{E}=\sup_{\xi\neq0}\|\xi\|^{-1}\|p(\xi)\|$.
Let $\mathcal{B}_\rho^0\subset\mathcal{E}^0$ be the complete subset of functions with Lipschitz constant at most $\rho$.
By the stable manifold theorem (\cite[Theorem 4.1]{chicone2006ordinary}), both the exact and approximate stable manifolds admit graph representations $p\in\mathcal{B}_\rho^0$ and $p^\eta\in\mathcal{B}_\rho^0$, with $Dp(0)=P$ and $Dp^\eta(0)=P^\eta$.

2. The Lyapunov--Perron method reduces stable manifold existence to fixed-point problems of integral operators:
$\bar p = \mathcal{LP}(\bar p)$ and $\bar p^\eta = \mathcal{LP}^\eta(\bar p^\eta)$,
where $\mathcal{LP},\mathcal{LP}^\eta: \mathcal{B}_\rho^0 \to \mathcal{B}_\rho^0$ are contraction mappings.
Specifically,
for the exact system \eqref{e:Hamiltonian-flow-2}, define
\begin{equation*}
\mathcal{LP}(\bar p)(\bar\xi) = -\int_0^\infty e^{t \Theta^\top} N_u\big(\bar x(t,\bar\xi,\bar p),\bar p(\bar x(t,\bar\xi,\bar p))\big) dt,
\end{equation*}
where $\bar x(t)$ solves
\begin{eqnarray}\label{e:bar-x-exact}
\dot{\bar x}(t) = \Theta \bar x(t) + N_s(\bar x(t), \bar p(\bar x(t))), \bar x(0) = \bar\xi,
\end{eqnarray}
and $\|e^{t\Theta^\top}\nu\|\le K e^{-bt}\|\nu\|$ for some $K,b>0$.
The nonlinear terms satisfy quadratic growth
\begin{eqnarray}\label{e:N_s-N_u-order}
\|N_s(\bar x,\bar p)\|, \|N_u(\bar x,\bar p)\| \le C_N\|(\bar x,\bar p)\|^2
\end{eqnarray}
and the local Lipschitz bound
\begin{align}\label{e:lip-con}
&\|N_u(\bar x_1,\bar p_1)-N_u(\bar x_2,\bar p_2)\|\\
&\le L_N\big(\|(\bar x_1,\bar p_1)\|+\|(\bar x_2,\bar p_2)\|\big)\|(\bar x_1-\bar x_2,\bar p_1-\bar p_2)\|,\notag
\end{align}
and likewise for $N_s$. Following the Lyapunov-Perron method (\cite[Chapter 4.1]{chicone2006ordinary}),
by Gronwall's inequality, there exists $\sigma_0>0$ (independent of $\varepsilon,\eta$) such that
$\mathcal{LP}$ is a contraction on $\mathcal{B}_\rho^0$ with factor $\kappa\in(0,1)$,
and trajectories satisfy $\|\bar x(t)\|,\|\bar p(t)\|\le K e^{-bt}\|\bar\xi\|$.

For the approximate system \eqref{e:hs-transform-app}, define
\begin{equation*}
\mathcal{LP}^\eta(\bar p)(\bar\xi) = -\int_0^\infty e^{t \Xi^\eta} N_u^\eta\big(\bar x^\eta(t,\bar\xi,\bar p),\bar p(\bar x^\eta(t,\bar\xi,\bar p))\big) dt,
\end{equation*}
with $\bar x^\eta$ solving
\begin{eqnarray}\label{e:bar-x-eta}
\dot{\bar x}^\eta(t) = \Theta^\eta \bar x^\eta(t) + N_s^\eta(\bar x^\eta(t), \bar p(\bar x^\eta(t))), \quad\bar x^\eta(0) = \bar\xi,
\end{eqnarray}
and $\|e^{t\Xi^\eta}\nu\|\le K e^{-b^\eta t}\|\nu\|$, $b^\eta\to b$ as $\eta\to0$.
Analogously, $\mathcal{LP}^\eta$ is a contraction with factor $\kappa^\eta\in(0,1)$. Moreover, the local Lipschitz bound
\begin{align}\label{e:lip-con}
&\|N_u^{\eta}(\bar x_1,\bar p_1)-N_u^{\eta}(\bar x_2,\bar p_2)\|\\
&\le L_{N^{\eta}}\big(\|(\bar x_1,\bar p_1)\|+\|(\bar x_2,\bar p_2)\|\big)\|(\bar x_1-\bar x_2,\bar p_1-\bar p_2)\|,\notag
\end{align}
and likewise for $N_s^{\eta}$.
By Lemmas \ref{l:app-condition}--\ref{l:F-error},
\begin{eqnarray}\label{e:Nu-eta-Nu}
\|N_u^\eta(\bar x,\bar p) - N_u(\bar x,\bar p)\| \leq C_N\eta \|(\bar x,\bar p)\|.
\end{eqnarray}
For small $\eta$, $L_{N^\eta}=L_N+O(\eta)$ and $\kappa^\eta=\kappa+O(\eta)<1$.

3.
To bound the fixed-point error $\|p^\eta - p\|_\mathcal{E}$, we first estimate the operator difference $\mathcal{LP}^\eta - \mathcal{LP}$. For any $\bar p \in \mathcal{B}_\rho^0$, decompose
\begin{align*}
&\mathcal{LP}^\eta(\bar p)-\mathcal{LP}(\bar p)
= \int_0^\infty \big(e^{t\Theta^\top}-e^{t\Xi^\eta}\big)N_u\,dt\\
&+ \int_0^\infty e^{t\Xi^\eta}\big(N_u-N_u^\eta\big)dt
+ \int_0^\infty e^{t\Xi^\eta}\big(N_u^\eta(\bar x)-N_u^\eta(\bar x^\eta)\big)dt\\
&\triangleq E_1+E_2+E_3.
\end{align*}
where all integrands are evaluated along the corresponding trajectories $(\bar x(t),\bar p(\bar x(t)))$ or $(\bar x^\eta(t),\bar p(\bar x^\eta(t)))$.

For $E_1$, Lemma \ref{l:F-error} yields that there exists $C_{E_1} > 0$ such that
$
\|e^{t \Theta^\top} - e^{t \Xi^{\eta}}\| \leq C_{E_1} \eta e^{-b t}
$
for $t \geq 0$. Using \eqref{e:N_s-N_u-order}, we have
$\|N_u(\bar x(t), \bar p(\bar x(t)))\| \leq C_N K^2 e^{-2b t} \|\bar\xi\|^2.$
 Therefore,
$
\|E_1\| \leq C_{E_1}C_N K^2 \eta \|\bar\xi\|^2 \int_0^\infty e^{-3bt} dt = \frac{C_{E_1}C_N K^2 \eta }{3b} \|\bar\xi\|^2.
$
Since $p \in \mathcal{B}_\rho^0$ and $\|\bar\xi\| \leq \sigma_0$ (within the local ball), we have
\begin{eqnarray}\label{e:E1}
\|E_1\|_\mathcal{E} \leq C_{E1}' \eta,
\end{eqnarray}
where $C_{E_1}' = \frac{C_{E_1}C_N K^2 }{3b}$.

For $E_2$, by the error bound \eqref{e:Nu-eta-Nu}, we have that
$
\|E_2\| \leq C_N K^2\eta \int_0^\infty e^{-2bt} \|\bar\xi\| dt = \frac{C_N K^2 \eta}{2b} \|\bar\xi\| \leq C_{E_2}' \eta \|\bar\xi\|,
$
where $C_{E_2}' = \frac{K^3 C_N}{2b}$. That is,
\begin{eqnarray}\label{e:E2}
\|E_2\|_\mathcal{E} \leq C_{E_2}' \eta.
\end{eqnarray}

For $E_3$, by Lemma \ref{l:F-error} and Appendix \ref{a:x-xeta}, we have that
\begin{eqnarray}\label{e:bar-x-error}
\|\bar x^\eta(t) - \bar x(t)\| \leq C_{E_3} \eta e^{-\frac{b}{2}t } \|\bar \xi\|
\end{eqnarray}
and the Lipschitz continuity of $N_u^\eta$:
\begin{eqnarray}
&&\|N_u^\eta(\bar x^\eta, \bar p(\bar x^\eta)) - N_u^\eta(\bar x, \bar p(\bar x))\| \notag\\
&\leq& L_{N^\eta} \left( \|\bar x^\eta -\bar x\| + \|\bar p(\bar x^\eta) - \bar p(\bar x)\| \right) \left( \|\bar x^\eta\| + \|\bar x\| \right).\notag
\end{eqnarray}
Substituting $\|\bar p(\bar x^\eta) - \bar p(\bar x)\| \leq \rho \|\bar x^\eta - \bar x\|$ and estimate \eqref{e:bar-x-error}, we obtain that
$
\|E_3\| \leq C_{E_3}' \eta \|\bar\xi\|.
$
That is,
\begin{eqnarray}\label{e:E3}
\|E_3\|_\mathcal{E} \leq C_{E_3}' \eta.
\end{eqnarray}

Combining the three error terms \eqref{e:E1}, \eqref{e:E2}, \eqref{e:E3}, there exists a constant $C_\Lambda = C_{E1}' + C_{E2}' + C_{E3}'$ such that
\begin{eqnarray}\label{e:lp-eta-lp}
\|\mathcal{LP}^\eta(\bar p) - \mathcal{LP}(\bar p)\|_\mathcal{E} \leq C_{\mathcal{LP}} \eta, \quad \forall \bar p \in \mathcal{B}_\rho^0.
\end{eqnarray}

4.
Let $\bar p=\mathcal{LP}(\bar p)$ and $\bar p^\eta=\mathcal{LP}^\eta(\bar p^\eta)$. By the triangle inequality and contractivity,
\begin{align*}
\|\bar p^\eta - \bar p\|_\mathcal{E}
&\le \|\mathcal{LP}^\eta(\bar p^\eta)-\mathcal{LP}(\bar p^\eta)\|_\mathcal{E}
+ \|\mathcal{LP}(\bar p^\eta)-\mathcal{LP}(\bar p)\|_\mathcal{E}\\
&\le C_{\mathcal{LP}}\eta + \kappa\|\bar p^\eta-\bar p\|_\mathcal{E},
\end{align*}
so $\|\bar p^\eta-\bar p\|_\mathcal{E}\le \frac{C_{\mathcal{LP}}\eta}{1-\kappa}$.
Choosing $\eta_0=\frac{(1-\kappa)\varepsilon}{C_{\mathcal{LP}}}$ gives $\|\bar p^\eta-\bar p\|_\mathcal{E}\le\varepsilon$.
Transforming back to original coordinates via the continuous, invertible maps $T,T^\eta$ yields \eqref{e:p-eta-p-local}.
\end{proof}

\subsection{Semiglobal estimate for the EDMD approximate stable manifold}\label{s:semiglobal}

We extend the local error estimate (Lemma \ref{l:p-error-local}) to the entire domain $\Omega$ via perturbation analysis of the ODE flows.

\textbf{Assumption 4.}
\emph{The bounded domain $\Omega$ contains the origin, and every $x\in\Omega$ is reachable from some initial point in $B_\rho(0)$ ($\rho\le\rho_0$) along the Hamiltonian characteristic flow \eqref{e:Hamiltonian-flow}.}

\textbf{Assumption 5.}
\emph{The stable manifold of the approximate characteristic system \eqref{e:Ham-flow-app} admits a $C^1$ graph representation $p^\eta(x)$ for all $x\in\Omega$.}

Let $F_H(Y)$ with $Y=(x,p)\in\mathbb{R}^{2n}$ denote the vector field of \eqref{e:Hamiltonian-flow}, with flow
$\Gamma(t,x,p(x)) = (X(t,x), P(t,x))$ for $x\in B_{\rho_0}(0)$,
where $\pi_x,\pi_p$ are the canonical projections onto $x$ and $p$ components, and $p(x)=\nabla V(x)$ is the gradient of the stabilizing HJB solution.
Correspondingly, let $F_H^\eta(Y)$ be the EDMD approximated vector field \eqref{e:Ham-flow-app}, with flow
$\Gamma^\eta(t,x,p^\eta(x)) = (X^\eta(t,x), P^\eta(t,x))$ for $x\in B_{\rho_0}(0)$.

For $\rho\in(0,\rho_0)$, define the forward maps
\begin{align*}
&X:\mathbb R_+\times B_{\rho}(0)\to \Omega, \quad (t,x)\to X(t,x),\\
&X^{\eta}:\mathbb R_+\times B_{\rho}(0)\to \Omega, \quad (t,x)\to X^{\eta}(t,x).
\end{align*}
For $T,r>0$ with $B_r(x_0)\subset B_{\rho_0}(0)$, denote the images of the ball under the two flows by
$X(T,B_r(x_0))$ and $X^\eta(T,B_r(x_0))$.

\begin{lemma}\label{l:local-error}
Suppose that Assumptions 2--5 hold. For any fixed $x_0\in B_{\rho_0}(0)$, let $r>0$ such that $B_r(x_0)\subset B_{\rho_0}(0)$, and let $T>0$. Then $X(T,B_r(x_0))$ and $X^{\eta}(T,B_r(x_0))$ are open, and for all $x\in B_r(x_0)$,
\begin{align}
\|X^{\eta}(T,x)-X(T,x)\| &\leq C(T)\eta,\label{e:X-eta-X}\\
\norm{P^\eta(T, x) - P(T, x)} &\leq C(T)\eta,\label{e:P-eta-P}
\end{align}
where $C(T)>0$ depends only on $T$.
\end{lemma}

\begin{proof}
1.
Since $F_H$ is smooth, the flow $\Gamma(t,\cdot,\cdot):\R^{2n}\to\R^{2n}$ is a $C^\infty$ diffeomorphism for each fixed $t$. Define $S(x)=\pi_x\circ\Gamma(T,x,p(x))$ for $x\in B_r(x_0)$. Since $p(x)=\nabla V(x)\in C^1$ (Remark~\ref{r:regularity}) and the embedding $x\mapsto(x,p(x))$ has full-rank Jacobian $\begin{bmatrix}I_n\\\nabla p(x)\end{bmatrix}$, the map $S$ is $C^1$. By the invariance of the stable manifold (Assumption~5),
$
\Gamma(T,x,p(x))=(X(T,x),\,p(X(T,x)))^\top, \forall x\in B_r(x_0),
$
and $\Gamma(T,B_r(x_0), p(B_r(x_0)))$ is open on the stable manifold.
The projection \(\pi_{x}:\mathbb R^{2n} \to \R^n\) is a surjective linear map. By the Open Mapping Theorem for finite-dimensional spaces, surjective linear maps map open sets to open sets. Hence $DS(x)$ has rank $n$, so $S$ is a local diffeomorphism and $S(B_r(x_0))=X(T,B_r(x_0))$ is open. The same argument applies to $X^\eta(T,B_r(x_0))$.

2.
Let $Y(t,x)=(X(t,x),P(t,x))$ and $Y^\eta(t,x)=(X^\eta(t,x),P^\eta(t,x))$ denote the flows of \eqref{e:Hamiltonian-flow} and \eqref{e:Ham-flow-app}, respectively. We write the approximate system as
\[
\dot{Y}^\eta = F_H(Y^\eta) + \bigl[F_H^\eta(Y^\eta)-F_H(Y^\eta)\bigr],\, Y^\eta(0,x)=(x,p^\eta(x)).
\]
Since $F_H$ is smooth on the compact set $\bar\Omega\times\bar\Omega_p$, it is Lipschitz with constant $L_H$. By Lemma~\ref{l:F-error}, $\norm{F_H^\eta(Y)-F_H(Y)}\leq C_{f,g,W}M\eta$, and by Lemma~\ref{l:p-error-local}, $\norm{Y^\eta(0,x)-Y(0,x)}=\norm{p^\eta(x)-p(x)}\leq C_p\eta$. Applying classical perturbation results for ODE (\cite[Theorem 3.4]{khalil2002nonlinear}) yields
\[
\norm{Y^\eta(T,x)-Y(T,x)}\leq \Bigl(C_p e^{LT}+\tfrac{C_{f,g,W}M}{L}(e^{LT}-1)\Bigr)\eta.
\]
Since $\norm{\pi_x}\leq 1$ and $\norm{\pi_p}\leq 1$, projecting onto the $x$- and $p$-components gives \eqref{e:X-eta-X}--\eqref{e:P-eta-P} with $C(T)=C_p e^{LT}+\frac{C_{f,g,W}M}{L}(e^{LT}-1)$.
\end{proof}

Based on the results obtained above, we are in a position to show that the EDMD approximate stable manifold is sufficiently close to the exact stable manifold.
\begin{theorem}\label{t:p-error}
Under Assumptions 2--5 with i.i.d. training samples, for any probabilistic tolerance $\delta\in(0,1)$ and any $\varepsilon>0$, there exists $\eta_0>0$ such that for all $\eta<\eta_0$,
\begin{eqnarray}\label{e:p-eta-p}
\|p^{\eta}(x)-p(x)\|< \varepsilon, \quad \forall x\in \Omega,
\end{eqnarray}
with probability at least $1-\delta$.
\end{theorem}
\begin{proof}
The theorem will be proved by the following steps.

1. By Assumption 4,
for any $x_1 \in \Omega$, there exists an initial point $x_0 \in B_{0.9\rho_0}(0)$ and time $T > 0$ such that the orbit $X(T,x_0) = x_1$. Let $B_r(x_0)\subset B_{\rho_0}(0)$ be any open ball. From Lemma \ref{l:local-error}, we have that $X(T,B_r(x_0))$ and $X^{\eta}(T,B_r(x_0))$ are open, and for all $x\in B_r(x_0)$,
\begin{eqnarray}\label{e:X-error}
\|X^{\eta}(T,x)-X(T,x)\|\leq C(T)\eta,
\end{eqnarray}
where $C(T)$ is a constant depending only on $T$.

2. We prove the existence of $x_0^\eta \in B_r(x_0)$ close to $x_0$ such that $X^\eta(T,x_0^\eta) = x_1$.
Indeed, by Lemma \ref{l:local-error},  $X(T,B_r(x_0))$ is open. Since $x_1 = X(T,x_0) \in X(T,B_r(x_0))$, there exists $r_1 > 0$ with $B_{r_1}(x_1) \subset X(T,B_r(x_0))$.
Choose $\eta \leq \eta_1$, where $\eta_1>0$ satisfies $C(T)\eta_1 < \frac{r_1}{4}$. Then from \eqref{e:X-error}, it holds that
$
\|X^\eta(T,x_0) - x_1\| \leq C(T)\eta_1 < \frac{r_1}{4}.
$
Since $X^\eta(T,B_r(x_0))$ is open, \eqref{e:X-error} yields
\begin{eqnarray}
x_1\in B_{\frac{r_1}{4}}(X^\eta(T,x_0))\subset X^\eta(T,B_r(x_0)).
\end{eqnarray}
Hence there exists $x_0^{\eta}\in B_r(x_0)$ such that $X^\eta(T,x_0^\eta) = x_1$.
Compute
\begin{align}\label{e:x0-eta-x0}
&\|x_0^\eta - x_0\| = \|X^{-1}(T,X(T,x_0^\eta)) - X^{-1}(T,x_1)\| \\
&\leq L_{X^{-1}} \|X(T,x_0^\eta) - X^\eta(T,x_0^\eta)\| \leq L_{X^{-1}} C(T)\eta.\notag
\end{align}
Here, the inverse map $X^{-1}(T,\cdot)$ of the Hamiltonian flow is Lipschitz continuous with Lipschitz constant $L_{X^{-1}}$.

3. We verify the estimate \eqref{e:p-eta-p} for fixed $x_1\in \Omega$. By invariance of the stable manifolds, $p(x_1) = \pi_p \Gamma(T,x_0,p(x_0))$ and $p^\eta(x_1) = \pi_p \Gamma^\eta(T,x_0^\eta,p^\eta(x_0^\eta))$.
We decompose
\begin{align}
  & \|p^\eta(x_1) - p(x_1)\| \\
  &\leq \|\pi_p \Gamma^\eta(T,x_0^\eta,p^\eta(x_0^\eta)) - \pi_p \Gamma(T,x_0^\eta,p^\eta(x_0^\eta))\|\notag\\
    &+ \|\pi_p \Gamma(T,x_0^\eta,p^\eta(x_0^\eta)) - \pi_p \Gamma(T,x_0,p(x_0))\|:=T_1+T_2.\notag
\end{align}

For $T_1$, by the perturbation result for ODE (\cite[Theorem 3.4]{khalil2002nonlinear}), it holds that
$
T_1\le \|\pi_p\|\|\Gamma^\eta(T,x_0^\eta,p^\eta(x_0^\eta)) - \Gamma(T,x_0^\eta,p^\eta(x_0^\eta))\| \leq C'(T)\eta,
$
where $C'(T)$ is constant depending only on $T$. Choosing $\eta \leq \eta_2$ with $C'(T)\eta_2 < \frac{\varepsilon}{2}$ gives $T_1 < \frac{\varepsilon}{2}$.

For $T_2$, estimate that
$
T_2\leq L_\Gamma \cdot \|(x_0^\eta,p^\eta(x_0^\eta)) - (x_0,p(x_0))\|,
$
where $L_\Gamma$ is the Lipschitz constant of $\Gamma(T,\cdot,\cdot)$.
Decompose
$
     \|(x_0^\eta,p^\eta(x_0^\eta)) - (x_0,p(x_0))\| \leq \|x_0^\eta - x_0\| +\|p^\eta(x_0^\eta) - p(x_0)\|.
$
Choosing $\eta \leq \eta_3$ with $L_{X^{-1}}C'(T)\eta_3 < \frac{\varepsilon}{4L_p}$ where $L_p$ is the Lipschitz constant of $p$,  by \eqref{e:x0-eta-x0},
\begin{eqnarray}\label{e:x0-eta-x0-2}
\|x_0^\eta - x_0\| \leq L_{X^{-1}}C(T)\eta< \varepsilon/(4L_p).
\end{eqnarray}

Moreover, by the triangle inequality,
$
\|p^\eta(x_0^\eta) - p(x_0)\| \leq \|p^\eta(x_0^\eta) - p(x_0^\eta)\| + \|p(x_0^\eta) - p(x_0)\|.
$
Since $x_0^\eta \in B_r(x_0) \subset B_{\rho_0}(0)$, by Lemma \ref{l:p-error-local}, it holds that there exists an $\eta_4>0$ such that for all $\eta<\eta_4$,
$
\|p^\eta(x_0^\eta) - p(x_0^\eta)\| \leq \varepsilon/4.
$
Using \eqref{e:x0-eta-x0-2}, we get
$
\|p(x_0^\eta) - p(x_0)\| \leq L_p \|x_0^\eta - x_0\| < \frac{\varepsilon}{4}.
$
Thus choosing $\eta \leq \min\{\eta_1,\eta_2,\eta_3,\eta_4\}$ makes $T_2 < \frac{\varepsilon}{2}$ and
$T_1 + T_2 < \varepsilon.$

4.
For each $x \in \bar{\Omega}$, by Steps 1-3 above, there exists an open neighborhood $U_x \subset \Omega$ and $\eta_x > 0$ such that for all $\eta \leq \eta_x$ and $y \in U_x$, $\|p^\eta(y) - p(y)\| \leq \varepsilon$.
Since $\bar{\Omega}$ is compact, by the Finite Covering Theorem, there exist finitely many neighborhoods $U_{x_1}, U_{x_2}, \dots, U_{x_k}$ such that $\bar{\Omega} \subset \bigcup_{i=1}^k U_{x_i}$.
Let $\eta_{\min} = \min\{\eta_{x_1}, \eta_{x_2}, \dots, \eta_{x_k}\}$. Then for all $\eta \leq \eta_{\min}$ and $x \in \bar{\Omega}$, $\|p^\eta(x) - p(x)\| \leq \varepsilon$.

Finally, by Proposition \ref{p:app} and Theorem \ref{t:app-l-ld}, there exists $d_0 \in \mathbb{N}$ such that for $d \geq d_0$, the corresponding $\eta \leq \eta_{\min}$, with the result holding with probability $1 - \delta$.
\end{proof}

\subsection{Estimate for the data-driven optimal control}\label{s:error-op}

This subsection analyzes the closed-loop stability of the EDMD approximate optimal control and estimates the corresponding cost error.

The exact optimal control is given by
\begin{eqnarray}\label{e:u*1}
u(x)=-W^{-1}g(x)^\top p(x),
\end{eqnarray}
and its EDMD approximation reads
\begin{eqnarray}\label{e:uNN}
u^{\eta}(x)=-W^{-1}g^{\eta}(x)^\top p^{\eta}(x).
\end{eqnarray}

\begin{lemma}[Error bounds for EDMD-approximated optimal control]\label{l:control-error}
Under Assumptions 2--5, for any $\varepsilon > 0$, there exist $\rho_0 > 0$ and $\eta_0 > 0$ such that for all $\eta \in [0, \eta_0)$,
\begin{align}
&\| u^\eta(x) - u(x) \| \leq \varepsilon,\quad \forall x\in\Omega,\label{e:control-error-global}\\
&\| u^\eta(x) - u(x) \| \leq \varepsilon \| x \|, \quad \forall x \in B_{\rho_0}(0).\label{e:control-error-local}
\end{align}
\end{lemma}
\begin{proof}
1. The control error can be written as
\begin{align}\label{e:control-error-decomp}
&\| u^\eta(x) - u(x) \| \\
&\leq \| W^{-1} \|  \| g^\eta(x)^\top p^\eta(x) - g(x)^\top p(x) \|.\notag
\end{align}
Triangle inequality yields
\begin{align}\label{e:term-decomp}
&\| g^\eta(x)^\top p^\eta(x) - g(x)^\top p(x) \| \\
&\leq \| g^\eta(x) \| \cdot \| p^\eta(x) - p(x) \| + \| g^\eta(x) - g(x) \| \cdot \| p(x) \|.\notag
\end{align}

2. We first prove the global error bound \eqref{e:control-error-global}. By Assumption 3 and Assumption 4, $g \in C^\infty(\bar{\Omega})$, $p \in C^1(\bar{\Omega})$. Thus, there exist constants $C_g > 0$ and $C_p > 0$ such that
$
\| g(x) \| \leq C_g, \| p(x) \| \leq C_p, \forall x \in \Omega.
$
By Theorem \ref{t:bounds-aug-general} and Assumption 2, there exists $\eta_1 > 0$ such that for all $\eta \in [0, \eta_1)$,
\begin{align}\label{e:g-eta-g}
\| g^\eta(x) - g(x) \| \leq \eta, \quad \forall x \in \Omega.
\end{align}
Letting $\eta \leq 1$, this implies
$
\| g^\eta(x) \| \leq \| g(x) \| + \| g^\eta(x) - g(x) \| \leq C_g + 1.
$
By Theorem \ref{t:p-error}, for any $\varepsilon > 0$, there exists $\eta_2 > 0$ such that for all $\eta \in [0, \eta_2)$,
$
\| p^\eta(x) - p(x) \| < \frac{\varepsilon}{2 \| W^{-1} \| (C_g + 1)}, \forall x \in \Omega.
$
Choose
$
\eta_0 = \min\left\{ \eta_1, \eta_2, \frac{\varepsilon}{2 \| W^{-1} \| C_p} \right\}.
$
Then, for all $\eta \in [0, \eta_0)$, \eqref{e:term-decomp} gives
$
\| g^\eta(x) \| \| p^\eta(x) - p(x) \| \leq (C_g + 1)\frac{\varepsilon}{2 \| W^{-1} \| (C_g + 1)} = \frac{\varepsilon}{2 \| W^{-1} \|},
$
and
$
\| g^\eta(x) - g(x) \| \| p(x) \| \leq \eta C_p \leq \frac{\varepsilon}{2 \| W^{-1} \| C_p}  C_p = \frac{\varepsilon}{2 \| W^{-1} \|}.
$
Combining these bounds with \eqref{e:control-error-decomp} yields \eqref{e:control-error-global}.

3. We now prove the local error bound \eqref{e:control-error-local}. By Lemma \ref{l:p-error-local}, for any $\varepsilon > 0$, there exist $\rho_0 > 0$ and $\eta_3 > 0$ such that for all $\eta \in [0, \eta_3)$,
$
\| p^\eta(x) - p(x) \| \leq \frac{\varepsilon}{2 \| W^{-1} \| (C_g + 1)} \| x \|,  \forall x \in B_{\rho_0}(0).
$
By Theorem \ref{t:stable-manifolds}, $p(x) = Px + O(\| x \|^2)$ near the origin. Thus, there exists a constant $C_p' > 0$ such that
$
\| p(x) \| \leq C_p' \| x \|, \forall x \in B_{\rho_0}(0).
$
Choosing
$
\eta_0' = \min\left\{ \eta_3, \eta_4, \frac{\varepsilon}{2 \| W^{-1} \| C_p'} \right\},
$
from \eqref{e:g-eta-g}, we have that for all $\eta \in [0, \eta_0')$,
$
\| g^\eta(x) \|  \| p^\eta(x) - p(x) \| \leq (C_g + 1)  \frac{\varepsilon}{2 \| W^{-1} \| (C_g + 1)} \| x \| = \frac{\varepsilon}{2 \| W^{-1} \|} \| x \|,
$
and
$
\| g^\eta(x) - g(x) \|  \| p(x) \| \leq \eta  C_p' \| x \| \leq \frac{\varepsilon}{2 \| W^{-1} \| C_p'} C_p' \| x \| = \frac{\varepsilon}{2 \| W^{-1} \|} \| x \|.
$
Combining these bounds with \eqref{e:control-error-decomp} yields \eqref{e:control-error-local}.

Finally, taking $\eta_0 = \min\{ \eta_0, \eta_0' \}$ completes the proof.
\end{proof}

Let $x(t)$ denote the solution of the exact closed-loop system
\begin{eqnarray}\label{e:x-closed}
\dot{x}=f(x)-g(x)u(x),\quad x(0)=x_0,
\end{eqnarray}
and $x^\eta(t)$ the solution of the EDMD closed-loop system
\begin{eqnarray}\label{e:x-closed-app}
\dot{x}=f(x)-g(x)u^{\eta}(x),\quad x(0)=x_0.
\end{eqnarray}
By Assumption $(\mathbf{A_{31}})$, $f,g,R$ are smooth and Lipschitz on $\bar\Omega$, and $p$ is Lipschitz on $\Omega$ (Remark \ref{r:regularity}).
\begin{theorem}\label{t:decay-xnn}
Under Assumptions 2--5, for any $\varepsilon>0$, there exists $\eta_0>0$ such that for all $\eta\in (0,\eta_0)$, the closed-loop trajectory $x^\eta(t)$ decays exponentially as $t\to\infty$, and
\begin{eqnarray}\label{e:x-eta-error}
\|x^{\eta}(t)-x(t)\|<\varepsilon, \quad \forall t\ge 0,
\end{eqnarray}
for all $x_0\in \Omega$. The cost error satisfies
\begin{eqnarray}
|J(x^{\eta},u^{\eta})-J(x,u)|<C_{f,g, W}\varepsilon
\end{eqnarray}
for some constant $C_{f,g,W}$ depending only on $f,g,W$.
\end{theorem}

\begin{proof}
The proof is analogous to \cite[Theorem III.2]{chen2024deep} and is omitted.
\end{proof}

\section{Algorithms}\label{s:algorithm}

\subsection{Constraint Approximated-SSD Algorithm for Augmented Control Systems}\label{s:A-SSD}

To identify approximate invariant subspaces of the Lie operator for the augmented system, we adopt the Approximated-SSD algorithm originally proposed in \cite[Section VII]{haseli2021learning}, with the strict constraint that the state observables $x=(x_1,\dots,x_n)$ and $y=(y_1,\dots,y_n)$ are retained in the subspace. These are essential for extracting system dynamics in subsequent steps. The goal is to obtain a low-dimensional approximate invariant subspace that supports reliable EDMD operator approximation and system analysis.

For the augmented control system \eqref{e:augment-system2}, we collect two types of i.i.d. samples: zero-control samples $D_0 = \{(\omega_j^0, \mathcal{L}^0\Psi(\omega_j^0))\}$ with $\omega_j^0\in\Omega$, and basis-control samples $D_i = \{(\omega_j^{E_i}, \mathcal{L}^{E_i}\Psi(\omega_j^{E_i}))\}$, where $E_i$ is the $i$-th standard unit control vector and $\mathcal{L}^{E_i}$ the corresponding basis-control Lie operator.

The observable dictionary $\mathbb{V} = \operatorname{span}\{\Psi(\xi)\}$ is constructed as specified in \eqref{e:observables}. All basis functions are at most first-order in the auxiliary variable $y$ and are linearly independent. The Approximated-SSD framework systematically prunes redundant basis functions while preserving the core observables required for dynamics extraction. For brevity, we do not reproduce the detailed algorithmic steps here; readers are referred to \cite[Section VII]{haseli2021learning} for the complete theoretical formulation and implementation guidelines of the Approximated-SSD framework.

\subsection{EDMD Algorithm for Augmented Control Systems}\label{s:EDMD-augmented}

We present an EDMD algorithm for the control-affine augmented system \eqref{e:augment-system2}, consisting of six steps outlined as follows:

\textit{Step 1 (Data collection):}
We collect two types of i.i.d. samples: zero-control samples \(D_0 = \{(\omega_j^0, F(\omega_j^0))\}\) and basis-control samples \(D_i = \{(\omega_j^{E_i}, F(\omega_j^{E_i}) + G(\omega_j^{E_i})E_i)\}\) for each unit control vector \(E_i\). These are merged into the training set \(X_{\text{train}} = \bigcup_{i=0}^m D_i\), and an independent validation set \(X_{\text{val}}\) is collected separately.

\textit{Step 2 (Dictionary construction):}
The dictionary follows specifications \eqref{e:observables}, with two core principles: (1) mandatory core observables \(x_1, \dots, x_n, y_1, \dots, y_n\) for state extractability; (2) extended observables are low-order polynomials (e.g., \(x_1^2, x_1y_1\)) that are at most first-order in \(y\).

\textit{Step 3 (Constrained approximate invariant subspace extraction):}
We extract an approximate Koopman-invariant subspace via null space analysis as explained in Section \ref{s:A-SSD}, with a critical constraint that core observables \(x=(x_1,\cdots,x_n)\) and \(y=(y_1,\cdots,y_n)\) are strictly retained (required for valid error bounds in Theorem \ref{t:bounds-aug-general} and Assumption 2).

\textit{Step 4 (EDMD Lie operator approximation):}
Using the reduced dictionary in Step 3, we split data into zero-control matrices \(X_0, Y_0\) and basis-control matrices \(X_{E_i}, Y_{E_i}\), then transpose all matrices to the standard EDMD format \((N_{d,\text{reduced}} \times d)\). We compute operators via Moore-Penrose pseudoinverse: \(\mathcal{L}_d^0 = Y_0 X_0^\dagger\), \(\mathcal{L}_d^{E_i} = Y_{E_i} X_{E_i}^\dagger\), and arbitrary control operator \(\mathcal{L}_d^U = \mathcal{L}_d^0 + \sum_{i=1}^M U_i(\mathcal{L}_d^{E_i} - \mathcal{L}_d^0)\) by linear superposition.

\textit{Step 5 (System component extraction):}
We use projection operators \(P_x \in \mathbb{R}^{n \times 2n}\) (extract \(x\)-component) and \(P_y \in \mathbb{R}^{n \times 2n}\) (extract \(y\)-component) to obtain:
(1) drift term: \(f^\eta(x) = P_x(\mathcal{L}_d^0 \cdot \Psi(x,y))\);
(2) control matrix: \(g_i^\eta(x) = P_x((\mathcal{L}_d^{E_i} - \mathcal{L}_d^0) \cdot \Psi(x,y))\);
(3) gradients: \(\theta^\eta(x)\) and \(\chi_i^\eta(x)\) which are approximations for $\nabla f$ and $\nabla g(x)$ in \eqref{e:aug-field} respectively.

\textit{Step 6 (Error verification and outputs):}
We verify four error indicators on \(X_{\text{val}}\): \(\alpha_1 = \|f-f^\eta\| \leq \eta\|x\|\), \(\alpha_2 = \|\nabla f-\theta^\eta\| \leq \eta\), \(\beta_1 = \|g-g^\eta\| \leq \eta\), \(\beta_2 = \|\nabla g-\chi^\eta\| \leq \eta\). We report mean and standard deviation, allowing minor overflows with warnings (core bounds hold with probability \(1-\delta\)). The algorithm outputs: (1) approximate invariant subspace results; (2) approximate Lie operator; (3) system components; (4) verification report.

\subsection{EDMD Stable Manifold}\label{s:algorithm-sm}

This section presents the stable manifold algorithm for the EDMD-approximated characteristic system \eqref{e:Ham-flow-app}, building upon the augmented system EDMD results established in Section \ref{s:EDMD-augmented}. Further details on deep learning-based stable manifold approximation can be found in \cite{chen2024deep} and \cite{chen2025h}.

\textit{Step 1. Building the EDMD-approximated characteristic System \eqref{e:Ham-flow-app}:}
From the algorithm in Section \ref{s:EDMD-augmented}, we extract system components $f^\eta(x)$, $g_i^\eta(x)$, $i\in[1:m]$, $\theta^\eta(x)$ and $\chi^\eta(x)$. Then we define auxiliary matrices $R^\eta(x)$ and $\tilde{R}_d(x)$, and construct the EDMD approximated characteristic system \eqref{e:Ham-flow-app}. By Lemma \ref{l:app-condition}, $\text{Ham}^\eta$ \eqref{e:Ham-mat-app} is hyperbolic for small $\eta$, guaranteeing a stable subspace.

\textit{Step 2. Deep learning for approximate stable manifold:}
For the EDMD-approximated characteristic system \eqref{e:Ham-flow-app}, we adapt the method from \cite[Section IV]{chen2024deep} with two modifications:

(1) \textit{Network architecture.} Starting from an initial network $p^{\text{NN}}_o(\theta,x)$, we enforce $p^{\text{NN}}(\theta,0)=0$ and $\partial_x p^{\text{NN}}(\theta,0)=P^\eta$ (the stabilizing solution to \eqref{e:general-Riccati}) via
\begin{equation}\label{e:NN-modified}
p^{\text{NN}}(\theta,x)
= p^{\text{NN}}_o(\theta,x) - p^{\text{NN}}_o(\theta,0)- \rho(x)\left[\frac{\partial p^{\text{NN}}_o}{\partial x}(\theta,0)x - P^{\eta}x\right],
\end{equation}
where $\rho(x)$ is a smooth cutoff function: $1$ for $|x|<r$, $0$ for $|x|\geq kr$, with $r>0$, $k>1$.

(2) \textit{Composite loss.} The loss combines mean and maximum prediction errors:
\begin{equation}\label{e:loss}
\mathcal L^\nu(\theta;\mathcal D)
:= \sigma_1 \cdot \frac{1}{|\mathcal D|}\sum_{i=1}^{|\mathcal D|}\|p_i-p^{\text{NN}}(\theta; x_i)\|^\nu
+ \sigma_2 \cdot \max_{i}\|p_i-p^{\text{NN}}(\theta; x_i)\|,
\end{equation}
with weights $\sigma_1,\sigma_2>0$ and $\nu\in[1,\infty]$.

\textit{Step 3. Computation of the approximate optimal control:}
Based on the trained NN $p^{\text{NN}}(\theta,x)$, the approximate optimal control is generated by \eqref{e:uNN}.

\section{Numerical examples}\label{s:example}
We validate the proposed method on a modified van der Pol system with a non-invariant dictionary space.

Consider the control-affine system
$
\dot{x} = f(x) + g(x)u,
$
with state $x=[x_1,x_2]^\top$, drift term
$f(x) = \begin{bmatrix} x_2 \\ \mu(1-x_1^2)x_2 - \sin(x_1) \end{bmatrix}$,
parameter $\mu=0.8$, control matrix $g(x)=[0,1]^\top$, and scalar input $u\in\mathbb{R}$.

To jointly approximate $f,g$ and their gradients (Theorems \ref{t:bounds-aug}--\ref{t:bounds-aug-general}), we introduce the augmented state $y=\dot{x}$ to construct the 4-dimensional augmented system \eqref{e:augment-system2} with $\xi=[x^\top,y^\top]^\top$ and augmented control $U=[u,v]^\top$. The gradients are
$$
\nabla f(x) = \begin{bmatrix} 0 & 1 \\ -\cos(x_1)-2\mu x_1 x_2 & \mu(1-x_1^2) \end{bmatrix},
\nabla g(x) = \begin{bmatrix} 0 & 0 \\ 0 & 0 \end{bmatrix}.
$$

\subsubsection{Data-driven approximation of the augmented system}
Training  data set are yielded by simulating the exact augmented system, corresponding to experimental measurements in practical scenarios.
States $x\in\Omega=\{(x_1,x_2)\mid x_1^2+x_2^2\le 4\}$ and auxiliary variables $y_1,y_2\in[-2,2]$ are uniformly sampled.
We adopt the time step $\Delta t=0.01$ and the DOP853 integrator to obtain one-step state pairs $(\xi,\xi(t+\Delta t))$.
A total of $7000$ samples are prepared for each control case, and the sample size may be raised for higher precision.
Three independent identically distributed datasets are gathered for $0=[0,0]$, $E_1=[1,0]$ and $E_2=[0,1]$ to approximate Lie operators $\mathcal{L}^0$, $\mathcal{L}^{E_1}$ and $\mathcal{L}^{E_2}$.

We construct a 30-dimensional polynomial observable set
\begin{eqnarray}\label{e:dic-1}
\Psi(x,y)=\{\Phi(x), \Phi(x)y_1, \Phi(x)y_2\},
\end{eqnarray}
where $\Phi(x)$ gathers all monomials of $x$ with degree no greater than three:
$$
\Phi(x)=\{1, x_1,x_2,x_1^2,x_1x_2, x_2^2, x_1^3,\cdots, x_2^3\}.
$$

The constrained approximated-SSD algorithm (Section \ref{s:A-SSD}) is executed on matrices $X,Y$ (cf. \eqref{e:X_0}-\eqref{e:Y_i}), while the essential observables $x_1,x_2,y_1,y_2$ are kept. The invariance error threshold is set to $\epsilon=1\times10^{-2}$. Based on the $7000\times 3$ training samples, the whole dictionary space serves as the approximate invariant subspace with invariance proximity $6.36\times10^{-3}$.
We compute Lie operator approximations via the Moore-Penrose pseudoinverse: $\mathcal{L}_d^0=Y^0(X^0)^\dagger$, $\mathcal{L}_d^{E_i}=Y^{E_i}(X^{E_i})^\dagger\;(i=1,2)$, and $\mathcal{L}_d^U = \mathcal{L}_d^0 + \sum_{i=1}^2 U_i (\mathcal{L}_d^{E_i} - \mathcal{L}_d^0)$ with $U=[U_1,U_2]$.

System components $f^\eta(x)$, $g^\eta(x)$, $\theta^\eta(x)$ and $\chi^\eta(x)$ are extracted from these operators.
A separate $1500\times 3$ sample validation dataset checks the approximation errors against $\eta=0.05$.The error standard deviation results read $\alpha_1=5.08\times 10^{-3}$, $\alpha_2=3.01\times 10^{-2}$, $\beta_1=1.22\times 10^{-3}$, $\beta_2=6.17\times 10^{-3}$.

\subsubsection{Deep learning for the approximate stable manifold}
We tackle the infinite horizon quadratic cost
$
J(x, u) = \int_0^\infty \left( \frac{1}{2}x^T Q x + \frac{1}{2}u^T W u \right) dt
$
with $Q=I_2$ and $W=I_1$.
The approximate Hamiltonian system \eqref{e:Ham-flow-app} is built from $f^\eta, g^{\eta}, \theta^\eta, \chi^{\eta}$.
Solving the approximate Riccati equation \eqref{e:general-Riccati} yields
$P^{\eta}=\left[
             \begin{array}{cc}
               2.98 & 0.434 \\
              0.412 & 2.38 \\
             \end{array}
           \right].$
The Hamiltonian matrix spectrum keeps a gap $0.788$ away from the imaginary axis, satisfying hyperbolicity.
Our deep learning scheme is then adopted for stable manifold approximation, which exhibits favourable sample efficiency and yields precise results with few training trajectories.

\textit{Generation of training and test sample sets:}
We investigate sample size influence by preparing two BVP datasets with 50 and 300 initial points uniformly sampled on the circle centered at $0$ with radius 0.8.
Each two-point BVP satisfying $x(0)=x_0,\,p(\infty)=0$ on $[0,30]$ is solved via \texttt{solve\_bvp}.
Converged trajectories are extended to $[-0.2,0]$ by the Radau scheme inside \texttt{solve\_ivp} to yield stable manifold orbits.
Points are sampled from each orbit: 5 uniform samples on $[-0.2,0]$ and 15 exponentially spaced samples on $[0,30]$ adapted to the decay rate $e^{-0.788t}$. Sample distributions are illustrated in Figure~\ref{f:domain}.
Twenty initial points are sampled following the identical rule to construct the held out test set for generalization assessment.

\begin{figure}[htbp]
\vspace{-0.3cm}
\centering
\subfigure{\includegraphics[width=0.45\textwidth]{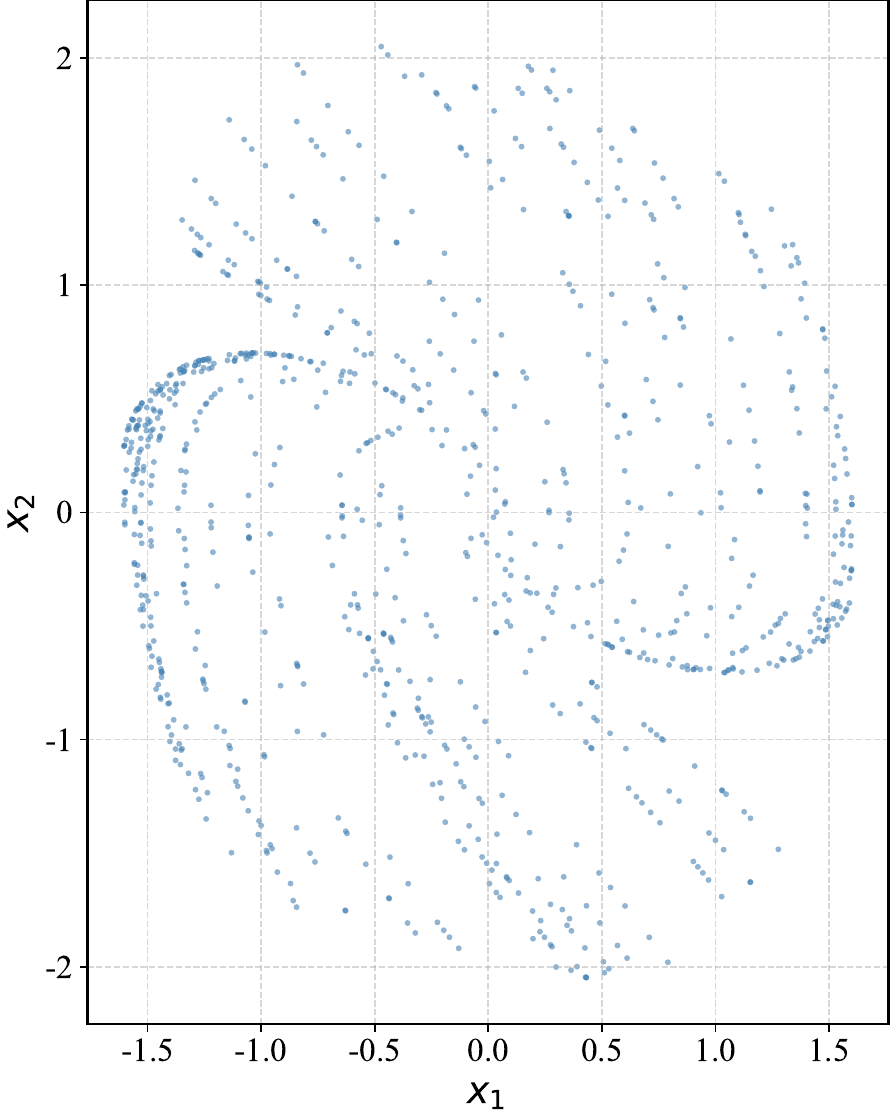}}
\subfigure{\includegraphics[width=0.45\textwidth]{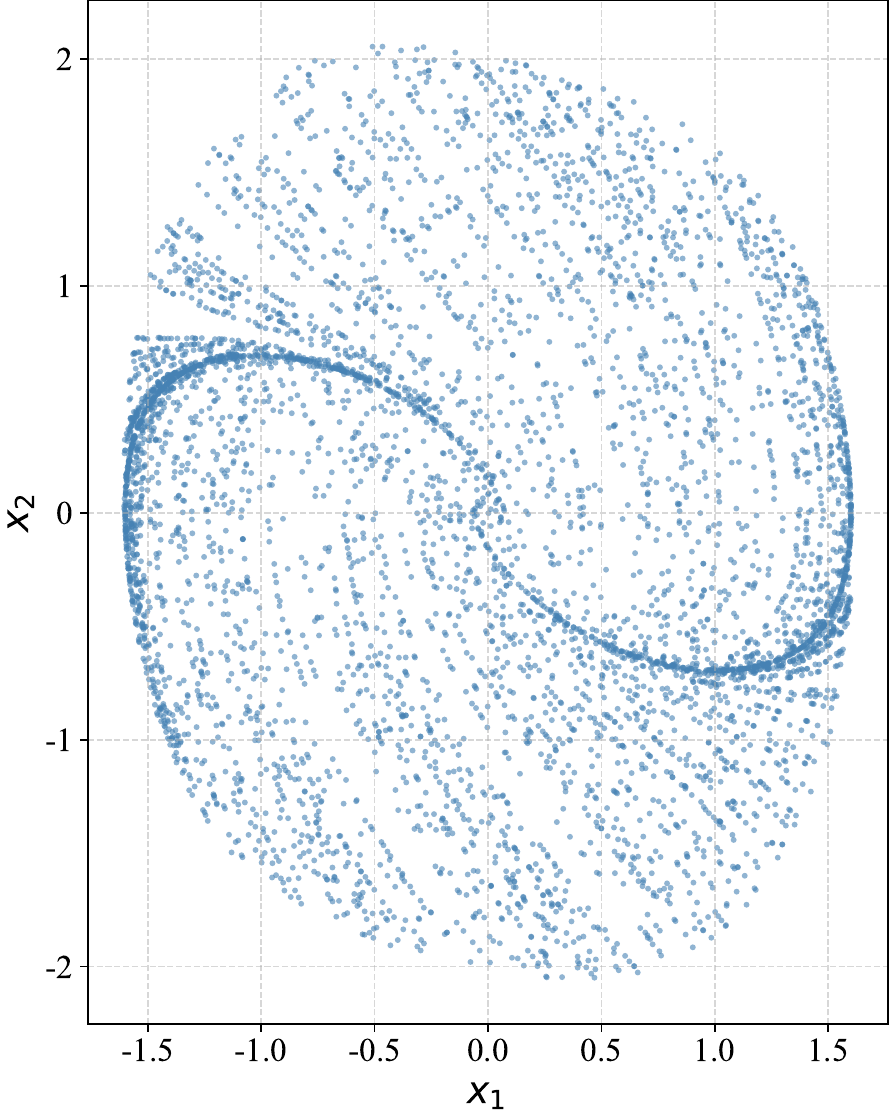}}
\vspace{-0.5cm}
\caption{Training sample distributions obtained from 50 and 300 trajectories.}
\label{f:domain}
\vspace{-0.0cm}
\end{figure}

\textit{Deep Neural Network Training:}
Identical network settings are adopted for comparative tests.
We build an LSTM network $p^{\text{NN}}_o(\theta,x)$ containing two 50 neuron hidden layers and modify it to $p^{\text{NN}}(\theta, x)$ via \eqref{e:NN-modified}, with cutoff function parameters $r=0.02,\,k=1.2$.
This 2-input-2-output network employs sine activation.

The composite loss \eqref{e:loss} uses weights $\sigma_1 = 1,\,\sigma_2 = 0.02$.
Training applies the Adam optimizer (initial learning rate $0.01$, halved every 500 epochs) over 3000 epochs, with $L_2$ gradient clipping at $5.0$.
All simulations run on an ordinary laptop (ThinkPad T480s without using GPU).
Training consumes around 60 s for the 50 trajectory dataset and 410 s for the 300 trajectory dataset.

The 50 trajectory case attains high precision: total loss $5.85\times 10^{-4}$, MSE $4.79\times 10^{-4}$ and maximum error $5.30\times 10^{-3}$.
Increasing trajectories to 300 yields only minor improvement (total loss $5.73\times 10^{-4}$, MSE $4.68\times 10^{-4}$, maximum error $5.22\times 10^{-3}$), showing near optimal performance under scarce training data.

We evaluate the two trained models on an independent test set to assess generalization performance.
The model trained on 50 trajectories achieves test total loss $7.79\times 10^{-4}$, MSE $5.72\times 10^{-4}$ and maximum error $1.04\times 10^{-2}$.
The model trained on 300 trajectories yields improved generalization, with test total loss $5.75\times 10^{-4}$, MSE $4.73\times 10^{-4}$ and maximum error $5.08\times 10^{-3}$. The minor train-test error gap verifies its reliable generalization and sufficient accuracy for stable-manifold-based control design under low-data conditions.

\subsubsection{Simulations}
We compute the feedback control law \eqref{e:uNN} from the trained network $p^{\text{NN}}(\theta, x)$ and carry out closed loop simulation \eqref{e:x-closed-app} over multiple initial states.
Three controllers are compared: our DD-SM approach trained on 50 and 300 trajectories, exact stable manifold (SM) optimal control and the data driven Riccati based LQR regulator.

Table \ref{tab:cost_comparison} lists the costs for representative starting points.
Both DD-SM configurations deliver costs close to the exact SM solution, and their state trajectories match the exact SM curve well, while the LQR response differs visibly (Figure \ref{f:comp1}).

\begin{table}[ht]
\centering
\caption{Cost comparison of different control methods}
\begin{tabular}{c|cccc}
\hline
Initial state $x_0$ & DD-SM (300) & DD-SM (50) & SM  & LQR \\
\hline
$(0.600, 0.000)$ & 0.5154 & 0.5148 & 0.5147 & 0.7194 \\
$(0.424, 0.424)$ & 0.5429 & 0.5421 & 0.542 & 0.7408 \\
$(0.000, 1.000)$ & 1.1816 & 1.1808 & 1.1805 & 1.7191 \\
$(-0.707, 0.707)$ & 1.0176 & 1.0194 & 1.0171 & 1.5303 \\
$(-1.400, 0.000)$ & 2.293 & 2.2942 & 2.3005 & 3.1319 \\
$(-0.990, -0.990)$ & 2.4842 & 2.4847 & 2.4858 & 3.4145 \\
$(-0.000, -1.800)$ & 3.6391 & 3.6395 & 3.6398 & 5.371 \\
$(1.273, -1.273)$ & 2.4653 & 2.4656 & 2.4651 & 3.781 \\
$(0.043, 1.199)$ & 1.6961 & 1.6953 & 1.6949 & 2.4645 \\
$(0.905, -1.196)$ & 2.0436 & 2.0435 & 2.0435 & 3.1058 \\
\hline
\end{tabular}
\label{tab:cost_comparison}
\end{table}

\begin{figure}[htbp]
\vspace{-0.3cm}
\centering
\subfigure{\includegraphics[width=0.45\textwidth]{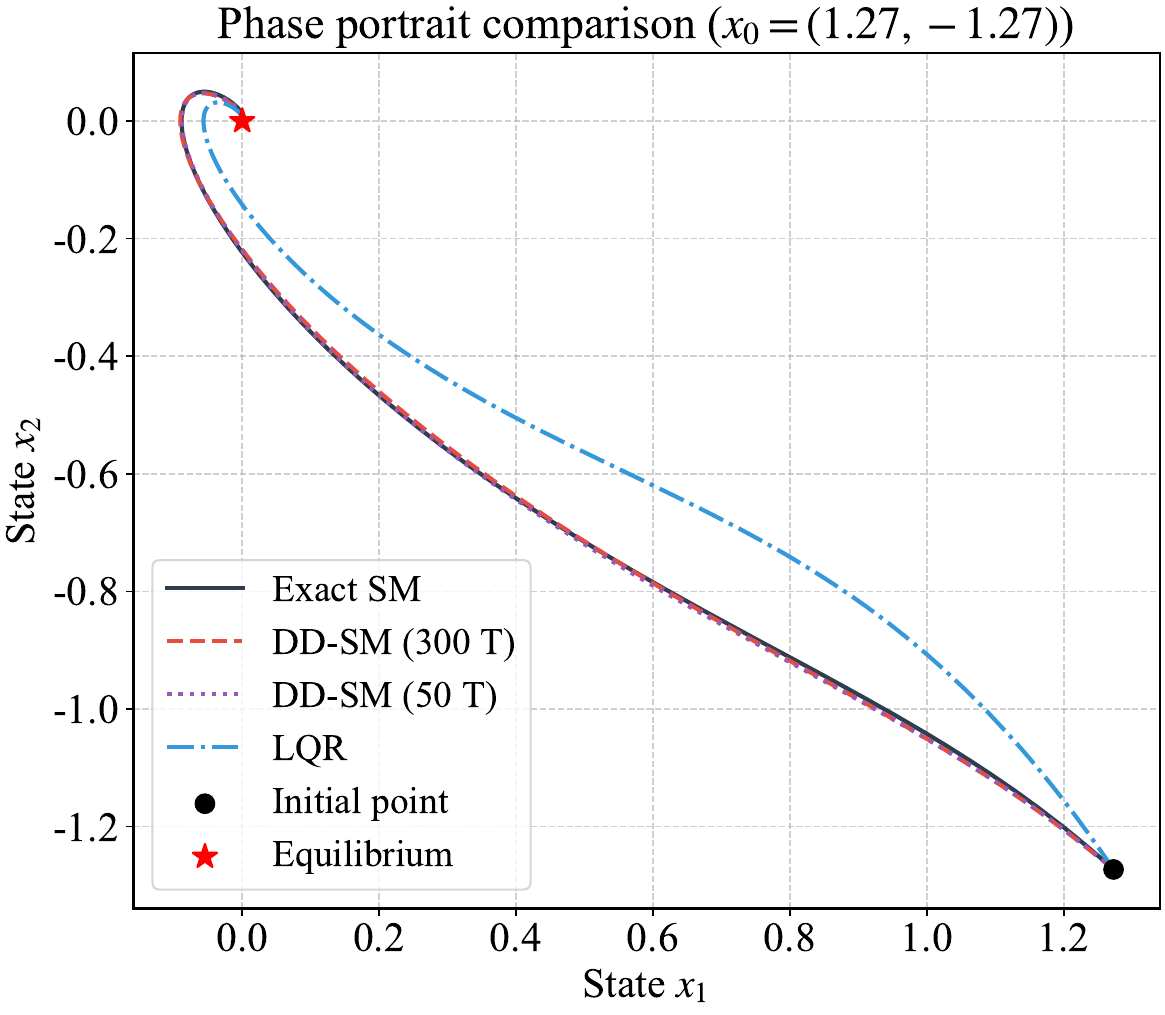}}
\subfigure{\includegraphics[width=0.45\textwidth]{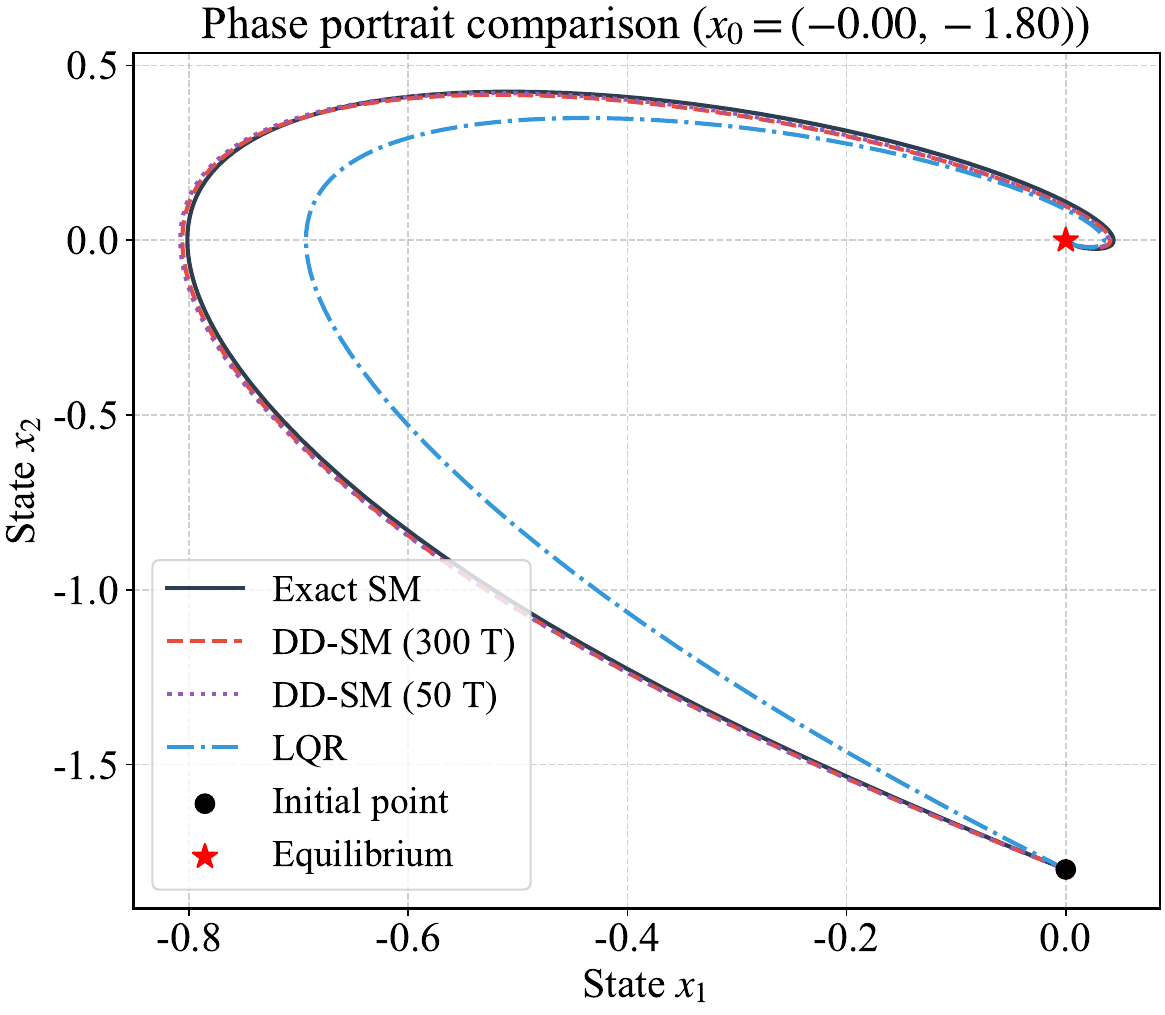}}
\vspace{-0.5cm}
\caption{Trajectory comparison: 50 and 300 trajectories DD-SM, exact SM and data-driven LQR}
\label{f:comp1}
\end{figure}

Each control signal is generated within one millisecond, satisfying real-time constraints.
\begin{figure}[htbp]
\vspace{-0.3cm}
\centering
\subfigure{\includegraphics[width=0.45\textwidth]{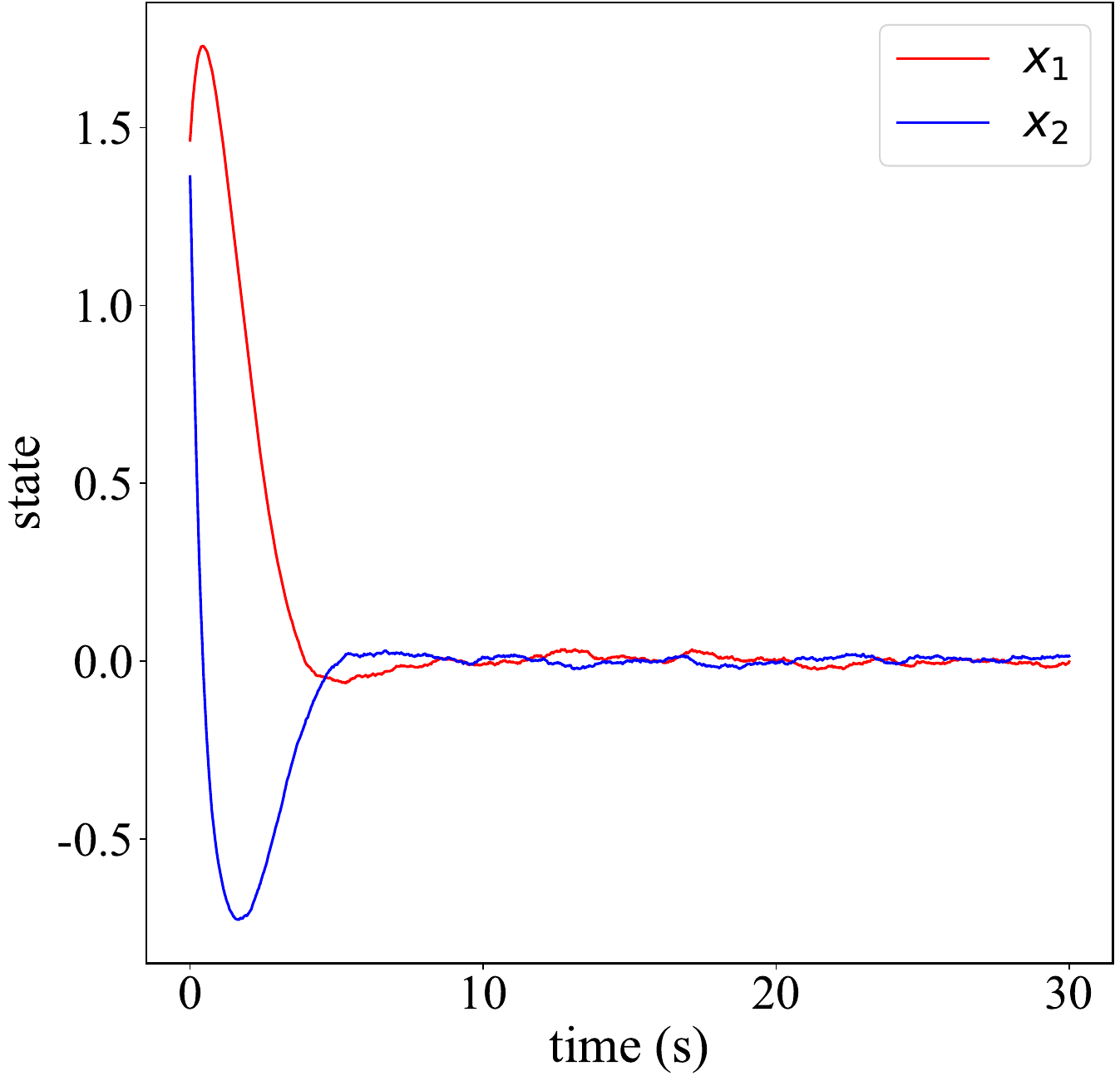}}
\subfigure{\includegraphics[width=0.45\textwidth]{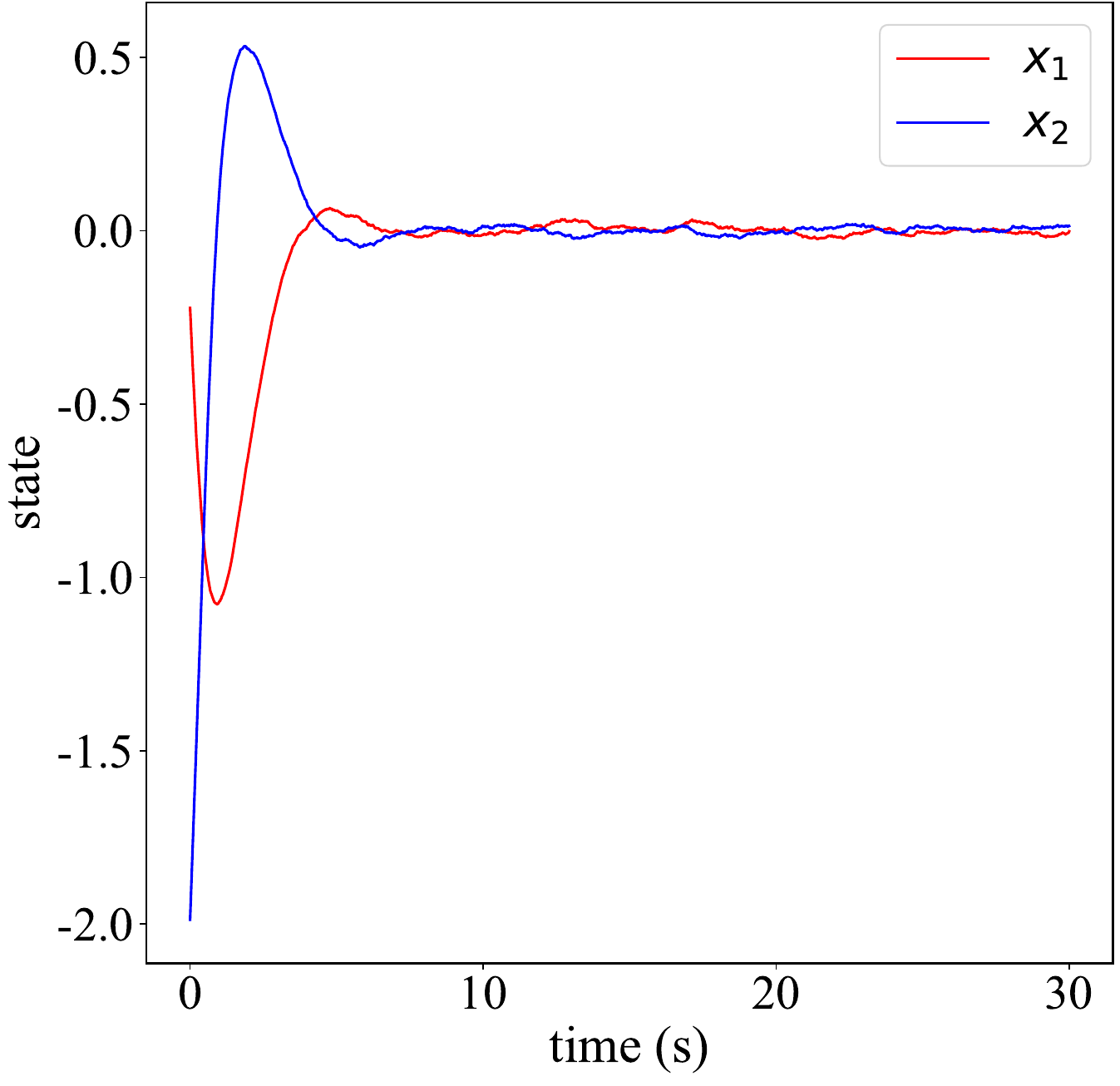}}
\vspace{-0.5cm}
\caption{Disturbance-rejection performance}
\label{f:robust}
\end{figure}
The controller possesses satisfactory robustness; trajectories can still be stabilized under 100-Hz white noise perturbation with standard deviation $0.01$ (Figure \ref{f:robust}).

\section{Conclusion}\label{s:conclusion}

Nonlinear optimal control with unknown dynamics is still challenging: existing Koopman based approaches lack rigorous links between modelling error and HJB closed loop stability as well as optimality. This work proposes the data driven stable manifold (DD-SM) method, which unites Koopman operator theory and the HJB stable manifold technique to synthesise optimal feedback control purely from trajectory measurements.

Our main results are four-fold.
First, an augmented control affine system is constructed, and an EDMD-SSD joint scheme approximates the drift, control matrix and their gradients. Probabilistic finite sample error bounds are derived for both invariant and non-invariant dictionaries.
Second, ODE and Lyapunov-Perron perturbation analysis proves the semi-global approximation error of the learned stable manifold can be arbitrarily small and decays with more training samples.
Third, relying on the manifold error estimates, we prove closed loop exponential stability and bound the optimality gap; better model precision improves the stability margin and cost performance. The above arguments build a complete theoretical chain from sample error to closed-loop guarantees. Finally, based on the theoretical results, we propose an efficient pipeline containing adaptive sampling and deep network fitting. Its inference latency is below one millisecond for real-time demands.
Tests upon a modified van der Pol oscillator demonstrate that the effectiveness of the DD-SM method.

\appendix

\subsection{Invariance Proximity for linear operators}\label{a:invariance-prox}
Given a Hilbert space linear operator $A:\mathbb H\to\mathbb H$ and finite dimensional subspace $\mathbb V\subset\mathbb H$, the invariance proximity index reads
$
\mathcal{I}_A(\mathbb V) = \sup_{\substack{v \in \mathbb V \\ \|Av\| \neq 0}} \frac{\|Av - \mathcal P_{\mathbb V } Av\|}{\|Av\|},
$
with $\mathcal P_{\mathbb V}$ the orthogonal projection onto $\mathbb V$.
This index obeys two basic properties: $\mathcal{I}_A(\mathbb V)=0$ exactly when $\mathbb V$ is $A$-invariant, and $0\le\mathcal{I}_A(\mathbb V)\le1$, where greater values stand for stronger invariance violation. The same proof as \cite[Theorem 5.1]{haseli2023invariance} yields the following result.
\begin{theorem}\label{t:index-sin}
Let $A\mathbb V = \{Av \mid v \in \mathbb V\}$ be the image of $\mathbb V$ under $A$, and let $0 \leq \theta_1 \leq \theta_2 \leq \cdots \leq \theta_m \leq \frac{\pi}{2}$ be the Jordan principal angles between $\mathbb V$ and $A\mathbb V$ (where $m = \dim(A\mathbb V)$). Then
$
\mathcal{I}_A(\mathbb V) = \sin\theta_m,
$
and there exists $v^* \in \mathbb V$ such that
$
\mathcal{I}_A(\mathbb V) = \frac{\|Av^* - \mathcal P_{\mathbb V }Av^*\|}{\|Av^*\|}.
$
\end{theorem}

%

\subsection{Estimation for $\bar x^{\eta}(t)-\bar x(t)$}\label{a:x-xeta}

Let $\Delta(t) = \bar x^\eta(t) - \bar x(t)$. Subtracting \eqref{e:bar-x-exact} from \eqref{e:bar-x-eta} gives
\begin{align}\label{e:delta-e}
\dot{\Delta}(t) &= \Theta \Delta(t) + (\Theta^\eta - \Theta) \bar x^\eta(t) + N_s^\eta(\bar x^\eta, \bar p(\bar x^\eta))\\
 &~~~- N_s(\bar x, \bar p(\bar x)), \quad \Delta(0) = 0.\notag
\end{align}
Note that $\|\Theta^\eta-\Theta\|\le C_\Theta\eta$ and $\|\bar x^\eta(t)\|\le K e^{-bt}\|\bar\xi\|$ imply
  \begin{eqnarray}\label{e:T_1}
  \|(\Theta^\eta - \Theta) \bar x^\eta(t)\| \leq C_\Theta K \eta e^{-bt} \|\bar\xi\|.
  \end{eqnarray}

For nonlinear perturbation term, letting $N_s^\eta(\bar x^\eta, \bar p(\bar x^\eta)) - N_s(\bar x^\eta,\bar p(\bar x^\eta))=T_1$ and $N_s(\bar x^\eta, \bar p(\bar x^\eta)) - N_s(\bar x, \bar p(\bar x))=T_2$, we
decompose
$
N_s^\eta(\bar x^\eta, \bar p(\bar x^\eta))-N_s(\bar x, \bar p(\bar x))=T_1+T_2.
$
Using Lemma \ref{l:app-condition}, \ref{l:p-app} and \ref{l:F-error}, we obtain that
$
\|N_s^\eta(\bar x,\bar p) - N_s(\bar x,\bar p)\| \leq C_N\eta \|(\bar x,\bar p)\|.
$
Hence, it holds that
$
\|T_1(t)\| \leq C_N \eta(1+L_{\bar p}) \|\bar x^\eta(t)\| \leq \hat C_N K \eta e^{-b t} \|\bar \xi\|.
$
Moreover, from \eqref{e:lip-con}, we have that
$
\|T_2(t)\| \leq 2 L_{N_s} K e^{-b t} \|\bar\xi\|\|\Delta(t)\|.
$
Hence we get
\begin{align}\label{e:T_2}
&\|N_s^\eta(\bar x^\eta, \bar p(\bar x^\eta))-N_s(\bar x, \bar p(\bar x))\| \\
&\leq C_N K \eta e^{-b t} \|\bar \xi\| + 2 L_{N_s} K e^{-b t} \|\bar\xi\|\|\Delta(t)\|.\notag
\end{align}

Substituting \eqref{e:T_1} and \eqref{e:T_2} into \eqref{e:delta-e}, multiplying by integrating factor $e^{-t \Theta}$, and integrating, we obtain that
$
\|e^{-t \Theta} \Delta(t)\| \leq K C_1 \eta \|\bar\xi\| t + 2 K L_{N_s} \|\bar\xi\| \int_0^t \|\Delta(s)\| ds,
$
where $C_1 = \max(C_{\Theta} K, C_N K)$. Using $\|e^{t \Theta}\| \leq K e^{-bt}$ and $te^{-bt}\le \frac{2}{be}e^{-\frac{b}{2}t}$, we get
\begin{eqnarray}\label{e:gronwall-1}
\|\Delta(t)\| \leq C_{2} \eta e^{-\frac{b}{2}t} \|\bar \xi\| + C_G \|\bar \xi\| e^{-bt} \int_0^t \|\Delta(s)\| ds,
\end{eqnarray}
with $C_{2} = \frac{2K C_1}{be}$, $C_G = 2 K L_{N_s}$. Define $z(t) = e^{\frac{b}{2} t} \|\Delta(t)\|$. Then \eqref{e:gronwall-1}
becomes
$
z(t) \leq C_{2} \eta \|\bar \xi\| + C_G \|\bar \xi\| e^{-\frac{b}{2}t} \int_0^t e^{-\frac{b}{2}s}z(s) ds.
$
By Gronwall's inequality,
$
z(t) \leq C_{2} \eta e^{\frac{2 C_G }{b}}\|\bar \xi\|^2.
$
Boundedness of $\Omega$ yields
$
\|x^\eta(t) - x(t)\| \leq C_{E_3} \eta e^{-\frac{b}{2}t} \|\bar \xi\|,
$
where $C_{E_3} > 0$ is independent of $\eta$, $\bar\xi$, and $t$.

\bibliographystyle{plain}

\end{document}